\documentclass{article}
\usepackage{iclr2027_conference,times}

\usepackage{amsmath,amssymb,amsthm,mathtools}

\usepackage{mathrsfs}
\usepackage{hyperref}
\usepackage{url}

\usepackage[capitalise]{cleveref}

\usepackage{pgfplots}
\usepackage{pgfplotstable}
\usepackage{mathtools}
\pgfplotsset{compat=1.18}
\usepackage{subcaption}

\definecolor{pepSimGDA}{RGB}{213,94,0}
\definecolor{behaviorGap}{RGB}{0,97,65}
\definecolor{lmiFull}{RGB}{0,114,178}
\pgfplotsset{
    pep figure axis/.style={
        width=\linewidth,
        height=0.72\linewidth,
        grid=major,
        grid style={draw=gray!20, line width=0.1pt},
        axis line style={black!70},
        tick style={black!70},
        tick label style={font=\footnotesize},
        label style={font=\footnotesize},
        legend style={font=\scriptsize, draw=black!45, fill=white,
                      cells={anchor=west}},
        legend pos=north east,
        mark options={solid},
        every axis plot/.append style={line width=0.9pt}
    }
}

\theoremstyle{plain}
\newtheorem{theorem}{Theorem}
\newtheorem{lemma}{Lemma}

\theoremstyle{remark}
\newtheorem{remark}{Remark}
\theoremstyle{definition}
\newtheorem{example}{Example}

\newcommand{\R}{\mathbb{R}}
\newcommand{\one}{\mathbf{1}}
\newcommand{\supp}{\operatorname{supp}}
\newcommand{\Gap}{\operatorname{Gap}}
\newcommand{\ip}[2]{\left\langle #1,#2\right\rangle}

\iclrfinalcopy

\title{AltGDA Achieves Global $O(1/T)$ Ergodic Convergence in Matrix Games}

\author{\parbox[t]{\dimexpr0.5\textwidth-2\tabcolsep\relax}{
  \normalfont\raggedright
  \textbf{Tianlong Nan}\\
  IEOR Department, Columbia University\\
  New York, NY 10027, USA\\
  \texttt{tianlong.nan@columbia.edu}
}
\And
\parbox[t]{\dimexpr0.5\textwidth-2\tabcolsep\relax}{
  \normalfont\raggedright
  \textbf{Garud Iyengar}\\
  IEOR Department, Columbia University\\
  New York, NY 10027, USA\\
  \texttt{garud@ieor.columbia.edu}
}
\AND
\parbox[t]{\dimexpr0.5\textwidth-2\tabcolsep\relax}{
  \normalfont\raggedright
  \textbf{Christian Kroer}\\
  IEOR Department, Columbia University\\
  New York, NY 10027, USA\\
  \texttt{christian.kroer@columbia.edu}
}
\And
\parbox[t]{\dimexpr0.5\textwidth-2\tabcolsep\relax}{
  \normalfont\raggedright
  \textbf{Shuvomoy Das Gupta}\\
  CMOR Department, Rice University\\
  Houston, TX 77005, USA\\
  \texttt{sd158@rice.edu}
}
}

\hypersetup{
  pdfborder={0 0 0},
  pdftitle={AltGDA Achieves Global O(1/T) Ergodic Convergence in Matrix Games},
  pdfauthor={Tianlong Nan, Garud Iyengar, Christian Kroer, Shuvomoy Das Gupta}
}

\begin{document}

\maketitle
\lhead{\(O(1/T)\) Ergodic Convergence of AltGDA}
\thispagestyle{plain}

\begin{abstract}
Alternating gradient descent--ascent (AltGDA) is a simple and practically effective method for solving finite two-player zero-sum matrix games. However, the theory of AltGDA remains limited: existing results either apply only to unconstrained settings or require restrictive assumptions on the equilibrium in constrained settings. We show that AltGDA converges globally at an $O(1/T)$ ergodic rate in every finite two-player zero-sum matrix game. Unlike prior results, our guarantee holds for every initialization and every horizon \(T\): the uniform averages of the AltGDA iterates satisfy an \(O(1/T)\) duality-gap bound. Our proof is inspired by numerical results obtained using a novel performance estimation programming (PEP) framework for Lyapunov function search over compact convex sets. Additionally, we provide simple counterexamples showing that the last-iterate duality gap of AltGDA does not converge to zero. This justifies why averaging of iterates is indeed necessary to achieve an \(O(1/T)\) rate. We have formalized and machine-checked our global ergodic convergence result in Lean 4.
\end{abstract}

\section{Introduction}
\label{sec:intro}

Game-theoretic learning algorithms underpin systems that attain superhuman play in poker~\citep{bowling2015heads,moravvcik2017deepstack,brown2018superhuman,brown2019superhuman}
as well as expert-level play in Stratego~\citep{perolat2022mastering}, and Diplomacy~\citep{meta2022human}.
A canonical model for learning in games is the two-player zero-sum matrix game, where two players select their strategies probabilistically to optimize an expected payoff, with one player minimizing and the other maximizing. 

The simplest learning dynamics is simultaneous gradient descent-ascent (SimGDA), where both players run projected gradient steps on their own strategies. SimGDA is known to achieve an $O(1/\sqrt{T})$ ergodic convergence rate in the duality gap. In this paper, we study the \emph{alternating} variant of gradient descent-ascent (AltGDA), where the players take turns updating their strategies, while always observing the newly-updated strategy of the opposing player. Empirically, AltGDA exhibits much better performance in solving finite two-player zero-sum matrix games compared to SimGDA, yet our theoretical understanding of AltGDA remains limited. Our question is whether the AltGDA dynamics, without any optimism or other tricks, admit a global $O(1/T)$ ergodic guarantee.

The question is delicate because plain gradient play in adversarial problems need not settle. In unconstrained bilinear games, simultaneous gradient dynamics diverges away from the equilibrium, and the alternating variant cycles~\citep{bailey2020finite}; 
on the simplex, adversarial regularized learning cycles as well~\citep{mertikopoulos2018cycles}. 
The classical remedies modify the dynamics through extragradient steps~\citep{korpelevich1976extragradient} 
or optimistic corrections~\citep{popov1980modification,rakhlin2013online}.  The stronger last-iterate guarantees known today are for those modified dynamics~\citep{wei2021last}.

Alternation itself has a long practical record. 
It entered large-scale game solving as a numerical device in Counterfactual Regret Minimization+ (CFR+)~\citep{tammelin2015solving}, the algorithm used for superhuman poker AI. Alternation provably does not hurt the theoretical guarantees~\citep{burch2019revisiting,farina2019online}. 
On random matrix games, numerical studies in \citep{kroer2020ieor8100_note5,nan2026convergence} suggest an ergodic $O(1/T)$ convergence rate for AltGDA. 
On the theory side, \citep{wibisono2022alternating,katona2024symplectic} show that alternation provably helps for alternating \emph{mirror} descent with Legendre regularizers. Legendre regularizers guarantee that the dynamics never touch the boundaries of the simplices.
In contrast AltGDA's Euclidean projections cause collisions with the simplex boundary, which is the main difficulty in a theoretical analysis.
\citet{nan2026convergence} sidestep this boundary issue by assuming that the game contains an interior equilibrium, or that the dynamics are initialized close to an equilibrium, and show $O(1/T)$ convergence for these special cases.

\paragraph{Contributions.}  Our contributions in this paper are threefold:

\begin{itemize}

\item \textbf{A global \(O(1/T)\) ergodic convergence result for AltGDA
  (\cref{sec:main-conv-altgda}, Appendix~\ref{sec:appendix-proof}).}
  We show that for every finite matrix game, every initialization, and every iteration budget \(T\), the uniform averages
  \((\overline x^T,\overline y^T)\) of the AltGDA iterates satisfy
  \(\Gap(\overline x^T,\overline y^T)\leq15/(2\eta T)\)
  for every fixed stepsize
  \(\eta \in (0, \tfrac{\delta}{2\sqrt2\,L}]\), where \(L>0\) bounds the
  spectral norm of the payoff matrix and \(\delta\in(0,1]\)
  is a positive constant associated with a Goldman-Tucker
  saddle point. 
  Our result provides the first global guarantee without
  any restrictions, showing that alternation alone suffices
  for an \(O(1/T)\) ergodic rate in general matrix games.
  We have formalized and machine-checked our convergence
  results in Lean~4 \citep{lean4}.

\item \textbf{A computer-assisted framework for Lyapunov-function discovery in constrained bilinear games
  (\cref{sec:lmi}, Appendix~\ref{sec:appendix-lmi}).}
  We present a computer-assisted framework based on performance estimation programming (PEP) that transforms the search for quadratic Lyapunov certificates for AltGDA into a linear matrix inequality problem (LMI). The framework combines computer-assisted automated Lyapunov analysis \citep{upadhyaya2025automated,upadhyaya2025autolyap}  with operator
  interpolation \citep{bousselmi2024interpolation} to construct an LMI over general compact convex bilinear problems, such that 
  exact feasibility of this LMI certifies a quadratic Lyapunov certificate leading to an \(O(1/T)\) ergodic
  convergence rate. We demonstrate numerically that this LMI is feasible subject to the tolerance of the SDP solver and the patterns present in computed numerical solutions guided us in deriving the analytical convergence results in the first contribution.

\item \textbf{The last-iterate duality gap of AltGDA does not converge to zero
  (\cref{sec:pep}, Appendix~\ref{sec:appendix-pep}).} Our first two contributions involve the ergodic duality gap, so a natural next question is if we can achieve an $O(1/T )$ convergence rate for AltGDA in last-iterate duality gap. We provide simple counterexamples showing that the last-iterate duality gap of AltGDA does not converge to zero both in matrix games and in bilinear problems over general compact convex sets. First, for compact convex sets, we show a simple $2$-dimensional example over the unit ball, where the last-iterate duality gap is exactly two for every iterate and every stepsize choice. Second, for the special case of matrix games, we present a \(2\times2\) matrix such that the last-iterate duality gap is bounded away from zero no matter how its stepsize is tuned. Our explicit counterexamples are motivated again by a PEP-based framework  that numerically computes the worst-case last-iterate duality gap of AltGDA over compact convex sets.

\end{itemize}

\paragraph{Julia and Lean code.}
For scientific reproducibility, we provide the Lean code for verifying our
global ergodic convergence result and the Julia code for reproducing our
numerical PEP results at:
\begin{center}
\url{https://github.com/Shuvomoy/AltGDA-code}
\end{center}
Unless otherwise specified, we conducted the experiments on a laptop running macOS Tahoe
with an Apple M5 Max processor and 128 GB of memory. The Lean verification 
uses Lean 4.32.0 and mathlib 4.32.0.

\section{Related Work}

\paragraph{Alternation in unconstrained minimax optimization.}
In unconstrained bilinear games,
\citet{bailey2020finite} showed that AltGDA with fixed stepsizes achieves constant alternating regret against any fixed comparator,
yielding \(O(1/T)\) average regret.
For smooth strongly convex-strongly concave objectives,
\citet{zhang2022near} established near-optimal local linear convergence
with better condition-number dependence than simultaneous updates.
\citet{lee2024fundamental} subsequently established a global advantage 
by proving an iteration-complexity upper bound for AltGDA below the
corresponding lower bound for simultaneous GDA.
\citet{feng2025continuous} showed that smaller heavy-ball momentum
improves local stability in continuous-time models.

\paragraph{Alternation on constrained strategy sets.}
For constrained bilinear games, \citet{wibisono2022alternating}
established an \(O(T^{-2/3})\) ergodic rate for alternating mirror
descent, which \citet{katona2024symplectic} improved to
\(O(T^{-4/5})\) through symplectic analysis.
These guarantees rely on Legendre regularizers and smoothness
assumptions that do not cover Euclidean projection onto simplices.
For projected AltGDA in matrix games, \citet{nan2026convergence}
proved an ergodic \(O(1/T)\) rate for sufficiently small constant
stepsizes, globally when an interior equilibrium exists and locally
without this condition.
Our theorem removes the interiority requirement, extending the global
guarantee to every finite matrix game and every initialization.
In online learning, \citet{cevher2023alternation} constructed
algorithms attaining \(\widetilde{O}(T^{1/3})\) alternating regret
for linear losses on simplices.
\citet{hait2025alternating} extended these guarantees to bounded
convex losses using continuous Hedge, with
\(\widetilde{O}(d^{2/3}T^{1/3})\) alternating regret in dimension \(d\).
The benefits of alternation depend on the learning rule:
\citet{lazarsfeld2025optimism} showed that, in fictitious play,
optimism can achieve constant regret without regularization,
whereas alternation alone need not.

\paragraph{Optimistic first-order methods in game solving.}
Usually, optimistic-type of modifications such as extrapolation and prediction are used in first-order methods for saddle-point optimization to achieve fast convergence rate.
Classical examples include the extragradient method~\citep{korpelevich1976extragradient}, 
Mirror-Prox~\citep{nemirovski2004prox}, 
dual extrapolation~\citep{nesterov2007dual}, 
and primal--dual splitting~\citep{chambolle2011first}.
The latter three achieve $O(1/T)$ ergodic convergence for various convex-concave problem classes.
In online learning, optimistic mirror descent and optimistic
follow-the-regularized-leader exploit predictability in the observed losses to obtain faster convergence in games~\citep{rakhlin2013optimization,syrgkanis2015fast}.
\citet{mokhtari2020unified} provide a unified analysis of extragradient and optimistic gradient descent--ascent (OGDA) as approximations of the proximal-point method.
A complementary line of work studies the last-iterate behavior of
optimistic methods~\citep{daskalakis2017training,mertikopoulos2018optimistic,daskalakis2018last}.
In particular, \citet{wei2020linear} establish linear last-iterate
convergence of OGDA in bilinear games over polytopes without assuming equilibrium uniqueness, and of optimistic multiplicative weights in matrix games with a unique equilibrium.
Finite-time last-iterate guarantees for extragradient and OGDA
also extend to constrained monotone games~\citep{cai2022finite}.
Our result shows that the $O(1/T)$ ergodic rate can also be attained through \emph{alternation alone}: 
the uniform average of the unmodified AltGDA iterates converges at the $O(1/T)$ rate in every finite matrix game, 
without additional correction, extrapolation, or prediction terms.

\paragraph{Performance estimation programming and Lyapunov analysis.}
The PEP framework was introduced by \citet{drori2014performance}
and developed through exact interpolation formulations by
\citet{taylor2017exact,taylor2017smooth}.
Subsequently, \citet{bousselmi2024interpolation} derived interpolation
conditions for linear operators, while
\citet{upadhyaya2025automated,upadhyaya2025autolyap} developed automated
Lyapunov analysis and corresponding software AutoLyap.
For AltGDA, \citet{nan2026convergence} used PEP to optimize stepsizes
against worst-case ergodic bounds. In contrast, we combine operator interpolation and automated Lyapunov analysis
to search for certificates, with the numerical results interpreted
through approximate feasibility in \cref{sec:lmi}.
We also study finite-horizon last-iterate performance over compact
convex sets in
\cref{sec:pep}.

\section{Preliminaries}
\label{sec:setting}

\paragraph{Bilinear saddle-point problems.}
In this paper we consider the following bilinear saddle point problems (SPP) of the form:
\begin{equation}
    \min_{x\in\mathcal X}\max_{y\in\mathcal Y}y^\top Ax,
    \label{eq:game-main} \tag{SPP}
\end{equation}
where \(\mathcal X\subseteq\R^n\) and \(\mathcal Y\subseteq\R^m\)
are nonempty compact convex sets and \(A\in\R^{m\times n}\).
Here \(\R^d\) denotes the \(d\)-dimensional Euclidean space.
We say that a feasible pair \((x^\star,y^\star) \in \mathcal{X} \times \mathcal{Y}\)  is a \emph{saddle point} if
\[
y^\top Ax^\star\leq(y^\star)^\top Ax^\star
\leq(y^\star)^\top Ax
\text{ for all }x\in\mathcal X,\ y\in\mathcal Y.
\]
Compactness, convexity, and bilinearity ensure that a saddle point
exists and that the two optimization orders have the common value
\[
\nu=\min_{x\in\mathcal X}\max_{y\in\mathcal Y}y^\top Ax
   =\max_{y\in\mathcal Y}\min_{x\in\mathcal X}y^\top Ax.
\]
For a feasible pair \((\widetilde x,\widetilde y) \in \mathcal{X} \times \mathcal{Y} \), define its
\emph{duality gap} by
\begin{align}
    \Gap(\widetilde x,\widetilde y)
    &:=\max_{y\in\mathcal Y}y^\top A\widetilde x
       -\min_{x\in\mathcal X}\widetilde y^\top Ax
      =\max_{\substack{x\in\mathcal X,\ y\in\mathcal Y}}
       \left(y^\top A\widetilde x-\widetilde y^\top Ax\right).
    \label{eq:gap-main}
\end{align}
This gap is always nonnegative on \(\mathcal X\times\mathcal Y\) and vanishes exactly at
saddle points.

\paragraph{Matrix games.}
Finite two-player zero-sum matrix games are the specialization of \eqref{eq:game-main} with 
\(\mathcal X=\Delta_n\), \(\mathcal Y=\Delta_m\), where
\(\Delta_d=\{z\in\R^d:z\geq0,\ \one^\top z=1\}\) is the
probability simplex. The \(x\)-player mixes over \(n\) actions to
minimize the expected payoff, and the \(y\)-player mixes over \(m\)
actions to maximize it. A saddle point is then a \emph{Nash equilibrium}
(NE); every finite matrix game has at least one~\citep{vonNeumann1928}.

\paragraph{Alternating GDA.}
Let \(\Pi_C\) denote Euclidean projection onto a nonempty closed
convex set \(C\). From an initialization
\((x^0,y^0)\in\mathcal X\times\mathcal Y\), alternating gradient
descent--ascent (AltGDA) with a fixed stepsize \(\eta>0\) generates the iterates
\begin{equation}
    \begin{aligned}
        x^{t+1}&=\Pi_{\mathcal X}\!\left(x^t-\eta A^\top y^t\right),\\
        y^{t+1}&=\Pi_{\mathcal Y}\!\left(y^t+\eta Ax^{t+1}\right),
    \end{aligned}
    \tag{AltGDA}
    \label{eq:altgda-main}
\end{equation}
for \(t=0,1,\ldots, T-1\), where $T \geq 1$ is the iteration budget. The \(x\)-update is a projected descent step;
the \(y\)-update is a projected ascent step using the newly computed
\(x^{t+1}\). At the time horizon \(T\), we distinguish the
last iterate \((x^T,y^T)\) from the uniform averages
\(\overline x^T=\frac1T\sum_{t=1}^T x^t\) and
\(\overline y^T=\frac1T\sum_{t=1}^T y^t\).

For the remainder of this section and the global ergodic convergence analysis, we specialize to \(\mathcal X=\Delta_n\) and
\(\mathcal Y=\Delta_m\) associated with \emph{AltGDA for matrix games}.

\paragraph{Goldman-Tucker strictly complementary saddle point.}
A saddle point is \emph{strictly complementary} if, 
at each coordinate,
either the equilibrium strategy has a strictly positive probability 
or the corresponding payoff slack is strictly positive. 
For a saddle point $(x^\star, y^\star)$, 
we define its payoff slack vectors as 
\begin{equation}
    s_x = A^\top y^\star - \nu\one \geq 0, \quad
    s_y = \nu\one-Ax^\star \geq 0.
    \label{eq:slacks-main}
\end{equation}
Then, $(x^\star, y^\star)$ is called a strictly complementary saddle point if 
\begin{equation}
    x_i^\star s_{x,i}=0, \quad 
    x_i^\star + s_{x,i}>0, \quad
    y_j^\star s_{y,j}=0, \quad 
    y_j^\star + s_{y,j}>0,
    \label{eq:strict-complementarity-main}
\end{equation}

for $i=1,2,\ldots,n$ and $j = 1,2,\ldots,m$. By the Goldman-Tucker theorem~\citep{GoldmanTucker1956,Williams1970}, 
every matrix game admits a strictly complementary saddle point. 
We fix one such pair and call it the \emph{Goldman-Tucker strictly complementary saddle point}. 

Let $L > 0$ be an upper bound of the spectral norm of $A$: $\lVert A\rVert_2 \leq L$.
Let $(x^\star, y^\star)$ be the Goldman-Tucker strictly complementary saddle point we choose, and its payoff slack vectors defined as in~\cref{eq:slacks-main}. 
Define the supports: 
$I := \supp(x^\star)$, 
$J := \supp(y^\star)$, and 
the smallest on-support probabilities: 
$\underline{x} := \min_{i \in I} x^\star_i > 0$ and 
$\underline{y} := \min_{j \in J} y^\star_j > 0$.
When the complements of the supports are nonempty, we denote the smallest off-support slacks by 
$\sigma_x := \min_{i \in I^c}s_{x,i}$ and $\sigma_y := \min_{j \in J^c}s_{y,j}$. 
The Goldman-Tucker theorem guarantees that $s_{x, i} > 0$ for all $i \in I^c$ and $s_{y, j} > 0$ for all $j \in J^c$. 
We then define the following scalar parameter: 
\begin{equation}
    \delta :=
    \min\left\{
      \underline x,\ \underline y,\
      \frac{\sigma_x}{L}\ \text{if }I^c\neq\varnothing,\
      \frac{\sigma_y}{L}\ \text{if }J^c\neq\varnothing
    \right\},
    \label{eq:delta-main}
\end{equation}
where terms corresponding to empty complements are omitted. 
It follows from the above that $0 < \delta \leq 1$.
We call $\delta$ the \emph{Goldman-Tucker separation constant} associated with the selected saddle point and the chosen bound \(L\).  
Intuitively, this constant records how decisively the Goldman-Tucker saddle point separates its supports from their complements, so that the larger the constant, 
the more decisive the separation (since all on-support coordinates are far from zero and all off-support coordinates have large payoff slacks).
Throughout this paper, we use \emph{on-support} and \emph{off-support} to indicate the set of coordinates corresponding to the selected Goldman-Tucker saddle-point.

\section{Main Theorem and its Proof}\label{sec:main-conv-altgda}

\subsection{Main theorem and discussion}

We establish the $O(1/T)$ ergodic convergence rate in the following theorem.
\begin{theorem}
\label{thm:main-text}
For any matrix game with the payoff matrix $A \in \R^{m\times n}$, 
let $L > 0$ satisfy $\lVert A \rVert_2\leq L$, 
and define $\delta$ from a selected Goldman-Tucker strictly complementary saddle point as in~\cref{eq:delta-main}. 
For every initialization
\((x^0,y^0)\in\Delta_n\times\Delta_m\), every \(T\geq1\), and every
fixed stepsize satisfying
$0 < \eta\leq\frac{\delta}{2\sqrt{2}\,L}$, 
the averages of the iterates generated by~\eqref{eq:altgda-main} for matrix games
with \(\mathcal X=\Delta_n\) and \(\mathcal Y=\Delta_m\) satisfy
\begin{equation}
    \Gap\left( \overline{x}^T, \overline{y}^T \right)
    \leq\frac{15}{2\eta T}.
    \label{eq:main-rate-main}
\end{equation}
\end{theorem}

We present a proof sketch of~\Cref{thm:main-text} in the next subsection, and defer the complete proof to Appendix~\ref{sec:appendix-proof}. 

\Cref{thm:main-text} establishes the first global $O(1/T)$ ergodic
convergence guarantee for AltGDA in \emph{every} finite matrix game.
Therefore, it closes the gap left by the previous theory of~\citet{nan2026convergence}: 
\emph{alternation alone, without optimism or any
additional correction step, is sufficient to achieve an $O(1/T)$ ergodic rate in general matrix games}.
More specifically, our result strengthens the previous guarantees in
the following respects:
\begin{itemize}
    \item \emph{No structural assumption on the game.}
    \Cref{thm:main-text} applies to every finite matrix game, and does not require the existence of an interior Nash equilibrium.
    Note that every matrix game admits a Goldman--Tucker strictly complementary saddle point due to the Goldman-Tucker theorem~\citep{GoldmanTucker1956,Williams1970}.
    In contrast, the previous global $O(1/T)$ guarantee required the game to admit an interior Nash equilibrium, 
    which is a restrictive assumption as many matrix games have equilibria supported only on proper faces of the simplices and admit no fully mixed equilibrium.

    \item \emph{No locality assumption on the initialization.}
    \Cref{thm:main-text} holds for every pair of initial strategies. 
    For games without an interior equilibrium, the previous $O(1/T)$ theorem required the initial strategies to lie in a sufficiently small neighborhood of a maximum-support equilibrium, which is a restrictive condition.

    \item \emph{A uniform theorem for every horizon.}
    For general matrix games, the previous evidence for a global $O(1/T)$ rate in terms of the ergodic duality gap came from numerical finite-horizon PEP computations for $T=5,6, \ldots, 50$ in \citet{nan2026convergence}, 
    which conjectured an $O(1/T)$ ergodic-rate.
    Our \Cref{thm:main-text} shows that the finite-horizon numerical conjecture in \citet{nan2026convergence} is indeed correct with an analytical proof that holds for every $T \geq 1$.
\end{itemize}

\subsection{Proof Sketch of~\Cref{thm:main-text}}\label{sec:architecture}

The standard path to an \(O(1/T)\) ergodic rate often follow a Lyapunov function based
argument: we show a nonnegative Lyapunov or energy function whose per-step decrease
dominates the quantity to be summed.  However, the presence of alternation and projection in AltGDA requires some careful modifications to the Lyapunov based argument. We briefly describe the  proof sketch of our main convergence result in five steps below with the details of each step shown in full in
Appendix~\ref{sec:appendix-proof}.  

\paragraph{Establish an exact energy balance.}

By the KKT conditions  \eqref{eq:altgda-main} for matrix games AltGDA admits the following equivalent representation without requiring explicit projection steps:
there are scalars $\gamma_t, \lambda_t$ and vectors $\mu_t\in\R_+^n,\rho_t\in\R_+^m$ such that (\cref{lem:kkt})
\begin{align}
    & \Delta x_t := x^{t+1} - x^t 
    =
    \eta(v_t-\gamma_t\one+\mu_t), \quad 
    \mu_t\odot x^{t+1}=0,
    \label{eq:kkt-x-main} \\
    & \Delta y_t := y^{t+1} - y^t
    =
    \eta(u_t-\lambda_t\one+\rho_t), \quad 
    \rho_t\odot y^{t+1}=0, 
    \label{eq:kkt-y-main}
\end{align}
where \(v_t := -A^\top y^t\) and \(u_t := Ax^{t+1}\), and \(\odot\) denotes the coordinatewise product.
We call \(\gamma_t,\lambda_t\) \emph{threshold scalars}, and \(\mu_t,\rho_t\) are \emph{multiplier vectors}. 

Let us fix a Goldman-Tucker saddle point $(x^\star, y^\star)$ as the selected NE.
We can then define several important quantities that are useful throughout the proof. 

We leverage an \emph{energy function} $V_t$ that consists of the squared Euclidean distance between the current iterates and the selected NE, plus a bilinear cross term to account for the alternation: 
\begin{equation}
    V_t :=
    \lVert x^t-x^\star\rVert_2^2
    +\lVert y^t-y^\star\rVert_2^2
    -\eta(y^t-y^\star)^\top A(x^t-x^\star).
    \label{eq:energy-main}
\end{equation}

We then define the \emph{dissipation} term $\mathscr{D}_t$ that scales with the magnitude of the mismatches between the selected NE and the current iterates on both on-support and off-support coordinates. 
\begin{equation}
    \mathscr{D}_t :=
    \eta\left( s_x^\top (x^t + x^{t+1}) + s_y^\top (y^t + y^{t+1} ) \right) + 2 \eta \left( \mu_t^\top x^\star + \rho_t^\top y^\star \right)
    \label{eq:dissipation-2}
\end{equation}
where the first term measures the probability masses of the current iterates on the off-support coordinates, 
and the second term measures the current iterates colliding with the simplex boundary on the on-support coordinates. 
We notice that both terms vanish when the supports of the selected NE are correctly identified. 

Finally, we define the \emph{residual} term $r_t := \eta v_t^\top\Delta x_t + \eta u_t^\top\Delta y_t - \lVert\Delta x_t\rVert_2^2 - \lVert\Delta y_t\rVert_2^2$. 
By using the KKT representations, we can show that the residual term is exactly (\cref{lem:residual})
\begin{equation}
    r_t = \eta(\mu_t^\top x^t+\rho_t^\top y^t) \geq 0. 
\end{equation}
This term is positive exactly when there exists at least one coordinate that has both positive probability mass in the previous iteration and a positive projection multiplier, and thus zero probability mass in the new iteration, by~\cref{eq:kkt-x-main,eq:kkt-y-main}.
Therefore, this term is positive only when the trajectories of~\eqref{eq:altgda-main} collide with the simplex boundary.

As the first step of the proof, we establish an exact identity that balances $V_t$, $\mathscr{D}_t$, and $r_t$ (\cref{lem:energy})
\begin{equation}
    V_{t+1} - V_t + \mathscr{D}_t = r_t.
    \label{eq:energy-balance-main}
\end{equation}
This equation explicitly describes the change of the energy function per step, and plays a central role in our proof.

\paragraph{Control the residual by the dissipation.}
The next central step we need is showing that the residual term is small relative to the dissipation term. In particular, we pursue an upper bound for the residual term $r_t$ which is smaller than $\mathscr{D}_t$ when we choose a small stepsize. 

To achieve this, we need to handle the on-support and off-support coordinates separately. 
The difficult part is the off-support block of coordinates. 
It is turns out that we can prove the following useful decomposition of the off-support block. For the $x$-block, letting $(v)_{I^c} \in \mathbb{R}^{| I^c |}$ denote the restriction of a vector $v$ to coordinates in $I^c$, we have (\cref{lem:reflection-x})
\begin{align}
    \eta(\mu_t)_{I^c}^\top (x^t)_{I^c} = - & \langle G_x (\Delta x_t)_{I^c}, (\Delta x_t)_{I^c} \rangle \notag \\ 
    & - \eta(d_t^x)^\top (\Delta x_t)_{I^c} - \eta \left( \tfrac{1}{| I |} \sum\nolimits_{i\in I}\mu_{t,i} \right) \left( \sum\nolimits_{i \in I^c} (\Delta x_t)_i \right), 
\end{align}
where $G_x$ is a positive semidefinite matrix and $d^x_t \in \mathbb{R}^{| I^c |}$ encodes the off-support payoff compared to the average of the on-support payoff.
Here, we present this decomposition using the same notations as before, 
and we write a complete statement and its proof with a set of more comprehensive notations in Appendix~\eqref{app:subsec:x-block}. 

For the on-support part, we have the following decomposition: 
\begin{equation*}
    \eta (\mu_t)_I^\top x_I^t =
    - \eta (\mu_t)_I^\top
    \left( (\Delta x_t)_I - ( \tfrac{1}{| I |}\sum\nolimits_{i \in I} (\Delta x_t)_i ) \one_{| I |} \right) + \eta \left( \tfrac{1}{| I |} \sum\nolimits_{i\in I}\mu_{t,i} \right) \!\! \left( \sum\nolimits_{i \in I^c} (\Delta x_t)_i \right).
    \label{eq:cancellation-x-main}
\end{equation*}

Combining both on-support and off-support components, the following upper bound holds for the $x$-block: 
\begin{equation}
    \begin{aligned}
    \eta\mu_t^\top x^t
    \leq{}
    B_{x,t}-B_{x,t+1} +
    \sqrt{1 + \tfrac{1}{|I|}}\,
    \tfrac{\eta^2 L^2}{\sigma_x}(s_x^\top (x^t + x^{t+1}))
    + \sqrt{1 - \tfrac{1}{| I |}}\,
    \tfrac{\eta^2 L}{\underline x}(\mu_t^\top x^\star), 
    \end{aligned}
    \label{eq:block-bound-x-main}
\end{equation}
where $B_{x,t} \coloneqq \eta(d_{t-1}^x)^\top(x^t)_{I^c}$, with
$d_{-1}^x=d_0^x$, as in~\cref{eq:storage-x}.
If $I^c=\varnothing$, set $B_{x,t}=0$ and omit the
$\sigma_x$-term in~\cref{eq:block-bound-x-main}.
\cref{eq:block-bound-x-main} bounds the $x$-block contribution
$\eta\mu_t^\top x^t$ to $r_t$ by the consecutive difference
of a bounded storage term plus $O(\eta)\mathscr{D}_t$,
where the implied constant depends only on the game and the selected
saddle point.
Combining this bound with the corresponding $y$-block bound
(\cref{eq:block-bound-y}), and setting
$B_t:=B_{x,t}+B_{y,t}$ with $B_{y,t}$ as in~\cref{eq:storage-y},
the stepsize condition in~\cref{thm:main-text} yields
\begin{equation}
    r_t \leq B_t - B_{t+1} + \tfrac{1}{2} \mathscr{D}_t.
    \label{eq:residual-upper-bounded-by-dissipation}
\end{equation}

\paragraph{Prove the residual is summable.}
The summability of the residual term follows directly from~\cref{eq:energy-balance-main,eq:residual-upper-bounded-by-dissipation}: 
\begin{equation*}
    \sum\nolimits_{t=0}^{T-1} r_t \leq B_0 - B_{T} + \frac{1}{2} \sum\nolimits_{t=0}^{T-1} \mathscr{D}_t = B_0 - B_{T} + \frac{1}{2} \sum\nolimits_{t=0}^{T-1} (r_t - V_{t+1} + V_t), 
\end{equation*}
which leads to a constant upper bound for $\sum_{t=0}^{T-1} r_t$ for any $T \geq 1$ because both $V_t$ and $B_t$ are bounded.
The rest of the proof follows the same path as that in~\citet[Lemma 1]{nan2026convergence}, which essentially proves the following descent inequality: 
\begin{equation}
    \eta (y^\top Ax^t - (y^t)^\top Ax) + \eta(y^\top Ax^{t+1}-(y^{t+1})^\top Ax) \leq \Phi_t - \Phi_{t+1}+r_t, \quad \forall\, t \geq 1, 
    \label{eq:combined-shifted-main}
\end{equation}
where $\Phi_t$ is a bounded Lyapunov function. 
Since the left-hand side of~\cref{eq:combined-shifted-main} is connected to the duality gap of the averaged iterates by averaging and maximizing, the summability of $\sum_{t=0}^{T-1} r_t$ is sufficient to prove an $O(1/T)$ convergence rate.

\section{A PEP Framework for Lyapunov Certificates}\label{sec:lmi}

Our Lyapunov-style analysis behind \cref{thm:main-text} was motivated by a PEP-based framework that numerically searches for quadratic Lyapunov certificates for AltGDA over general compact convex sets. Our framework is inspired by the recent automated PEP-based Lyapunov analysis of \citet{upadhyaya2025automated,upadhyaya2025autolyap} and the operator
interpolation results of \citet{bousselmi2024interpolation}.  

The core idea behind our PEP-based Lyapunov framework is modeling and solving a linear matrix inequality (LMI) over the space of appropriate quadratic Lyapunov functions, and if this LMI is feasible, then there exists a suitable Lyapunov function that can lead to the desired ergodic convergence rate for AltGDA. In our Lyapunov function search three constraints (denoted by $C_1, C_2, C_3$) are modeled in the LMI: constraint \(C_1\) requires one-step Lyapunov descent up to an auxiliary quadratic term, \(C_2\) requires a
nonnegative potential, and \(C_3\) compares the auxiliary term with the
one-step comparator gap, with the decision variables being two symmetric Lyapunov
matrices, an auxiliary quadratic pair, and condition-specific positive
semidefinite (PSD) blocks with interpolation and normal-cone
multipliers. Also rather than considering only one iteration of AltGDA, we consider a variable number of iterations modeled by a history parameter \(h\), as proving the desirable convergence rate may require more than one iteration. Here $h=0$ corresponds to one iteration of AltGDA, $h=1$ means two iterations and so on. 

An \emph{exactly} feasible solution to the Lyapunov LMI satisfies \(C_1, C_2, C_3\) on every
window, and telescoping gives the resultant $O(1/T)$ ergodic convergence rate for AltGDA in \cref{thm:c123-lmi-conditional-rate} below. However, the standard SDP solvers can only compute solutions with a numerical tolerance (e.g., Mosek \citep{mosekjl} has a default feasibility tolerance of $10^{-8}$). Such a numerically computed feasible solution is only \emph{approximate} and \emph{not exact}. However, the structure present in the approximately feasible solution to this PEP-based LMI often reveals specific patterns \citep{goujaud2023fundamental}, which we can then exploit to arrive at an analytical Lyapunov function for AltGDA leading to the desired $O(1/T)$ ergodic convergence rate.
In Appendix~\ref{sec:appendix-lmi}, we describe the formulation for  the LMI  in detail and also discuss how the feasibility of this LMI leads to an $O(1/T)$ convergence rate for AltGDA. For brevity we just present the core result below.

\begin{theorem}[Conditional Lyapunov--LMI rate]
\label{thm:c123-lmi-conditional-rate}
Let \(A\in\R^{m\times n}\) satisfy \(\lVert A\rVert_2\leq1\), and let
\(\mathcal X\subseteq\R^n\)
and \(\mathcal Y\subseteq\R^m\) be compact convex sets with
\(\lVert x\rVert_2\leq D\) for \(x\in\mathcal X\) and
\(\lVert y\rVert_2\leq D\) for \(y\in\mathcal Y\).  Let
\(\{(x^t,y^t)\}_{t\geq0}\) be generated by the AltGDA iteration \eqref{eq:altgda-main} with initialization
\((x^0,y^0)\in\mathcal X\times\mathcal Y\) and stepsize \(\eta>0\), and
let \(\Gap\) be defined by \eqref{eq:gap-main}. Fix an integer \(h\geq0\).  Suppose the Lyapunov LMI of
Appendix~\ref{sec:appendix-lmi} is exactly feasible for \((h,\eta)\), and fix
any feasible certificate.  Let \(Q^x\) and \(Q^y\) be the Lyapunov
matrices in that certificate, and define
\(\varsigma_x^+=\max\{\lambda_{\max}(Q^x),0\}\) and
\(\varsigma_y^+=\max\{\lambda_{\max}(Q^y),0\}\).  Fix an integer
\(T\geq1\).  Then
\[
    \Gap\left(
        \frac1T\sum_{k=0}^{T-1}x^k,\;
        \frac1T\sum_{k=0}^{T-1}y^k
    \right)
    \leq
    \frac{D^2\left((h+4)\varsigma_x^++(h+5)\varsigma_y^+
    +(h+1)\eta^2\left(\varsigma_x^++\varsigma_y^+\right)\right)}{T}.
\]
\end{theorem}

\paragraph{Numerical results.}

In our numerical experiments, for a numerical solution to the Lyapunov LMI to be considered approximately feasible, we require: 
\begin{itemize}
    \item MOSEK reports \texttt{OPTIMAL} and \texttt{FEASIBLE\_POINT}.
    \item The largest absolute equality residual is less than or equal to $10^{-6}$.
    \item The smallest eigenvalue among the required PSD blocks is $-10^{-7}$ or larger.
    \item The smallest nonnegativity-constrained dual multiplier is $-10^{-8}$ or larger.
\end{itemize}

Figure~\ref{fig:c123-mosek-lmi-diagnostics} summarizes the numerical
Lyapunov search over the stepsize \(\eta\) and history parameter \(h\). We find that for $h=0, 1$, the LMI in consideration does not yield an approximately feasible solution based on our acceptance criteria. Starting from $h \geq 2$, we find that MOSEK can find an approximately feasible solution to the LMI for a range of stepsizes, as shown in \Cref{fig:c123-mosek-lmi-diagnostics}. \Cref{fig:c123-mosek-lmi-diagnostics}(a) shows the range of accepted stepsizes. On the other hand, \Cref{fig:c123-mosek-lmi-diagnostics}(b) shows how closely the computed
candidates satisfy the LMI constraints based on a diagnostic score that we define as follows. Let \(r\) be the maximum absolute affine-identity residual,
\(e\) the minimum eigenvalue over the required PSD blocks, and
\(\ell\) the minimum nonnegativity-constrained multiplier. Then the diagnostic score for a numerical solution is given by:
\[
s=\max\left\{
  \frac{r}{10^{-6}},
  \frac{(-10^{-7}-e)_+}{10^{-7}},
  \frac{(-10^{-8}-\ell)_+}{10^{-8}}
\right\},
\text{ where } (a)_+=\max\{a,0\},
\]

where Figure~\ref{fig:c123-mosek-lmi-diagnostics}(b)  displays
\(\log_{10}(\max\{s,10^{-16}\})\). All grids with dark green, light green, or yellow-green color satisfy the
aforementioned numerical acceptance tests. The gray grids show the pairs that failed the solver-status requirement.

\begin{figure}[t]
    \centering
    \begin{subfigure}[t]{0.48\linewidth}
        \centering
        \begin{tikzpicture}
            \begin{axis}[
                pep figure axis,
                xlabel={$h$},
                ylabel={Accepted range of $\eta$},
                xtick={2,3,4,5},
                ymode=log,
                log basis y=10,
                ymin=8e-4,
                ymax=0.7,
                ytick={0.001,0.01,0.1,0.5},
                yticklabels={$10^{-3}$,$10^{-2}$,$10^{-1}$,$0.5$},
                enlarge x limits=0.22
            ]
                \addplot[
                    draw=none,
                    mark=none,
                    error bars/y dir=both,
                    error bars/y explicit,
                    error bars/error bar style={draw=lmiFull, line width=3.2pt},
                    error bars/error mark options={draw=lmiFull, line width=0.8pt}
                ]
                    table[
                        col sep=comma,
                        x=h,
                        y=eta_midpoint,
                        y error=eta_radius
                    ] {data/c123_lmi_optimal_eta_ranges.csv};
            \end{axis}
        \end{tikzpicture}
        \caption{History parameter $h$}
    \end{subfigure}
    \hfill
    \begin{subfigure}[t]{0.48\linewidth}
        \centering
        \begin{tikzpicture}
            \begin{axis}[
                pep figure axis,
                width=0.84\linewidth,
                xlabel={$\eta$},
                ylabel={$h$},
                xmode=log,
                log basis x=10,
                xmin=8e-4,
                xmax=0.75,
                ytick={2,3,4,5},
                ymin=1.55,
                ymax=5.45,
                colormap={diagnosticmap}{
                    color(0cm)=(behaviorGap!75!black);
                    color(1cm)=(yellow!80!black);
                    color(2cm)=(pepSimGDA!90!black)
                },
                point meta min=-3,
                point meta max=0,
                colorbar,
                colorbar style={
                    width=5pt,
                    ylabel={$\log_{10}$ diagnostic score},
                    ylabel style={font=\scriptsize},
                    yticklabel style={font=\scriptsize}
                }
            ]
                \addplot[
                    only marks,
                    mark=square*,
                    mark size=2.2pt,
                    mark options={draw=gray!60, fill=gray!25},
                    restrict expr to domain={\thisrow{strict_mosek_primal_point_id}}{0:0}
                ]
                    table[
                        col sep=comma,
                        x=eta,
                        y=h
                    ] {data/c123_lmi_optimal_eta_heatmap.csv};
                \addplot[
                    scatter,
                    only marks,
                    mark=square*,
                    mark size=2.2pt,
                    scatter/use mapped color={draw=mapped color, fill=mapped color},
                    point meta=\thisrow{log10_diagnostic_score},
                    restrict expr to domain={\thisrow{strict_mosek_primal_point_id}}{1:1}
                ]
                    table[
                        col sep=comma,
                        x=eta,
                        y=h
                    ] {data/c123_lmi_optimal_eta_heatmap.csv};
            \end{axis}
        \end{tikzpicture}
        \caption{Diagnostic score over $\eta$}
    \end{subfigure}
    \caption{Numerical Lyapunov search for AltGDA.
(a) Range of accepted stepsizes yielding approximately feasible solutions for the LMI for each history parameter \(h \geq 2\). Note that for $h=0, 1$, Mosek could not find any feasible solution.
(b) Quality of the numerical solutions for \((h,\eta)\) with darker green indicating
higher precision solutions.}
\label{fig:c123-mosek-lmi-diagnostics}
\end{figure}
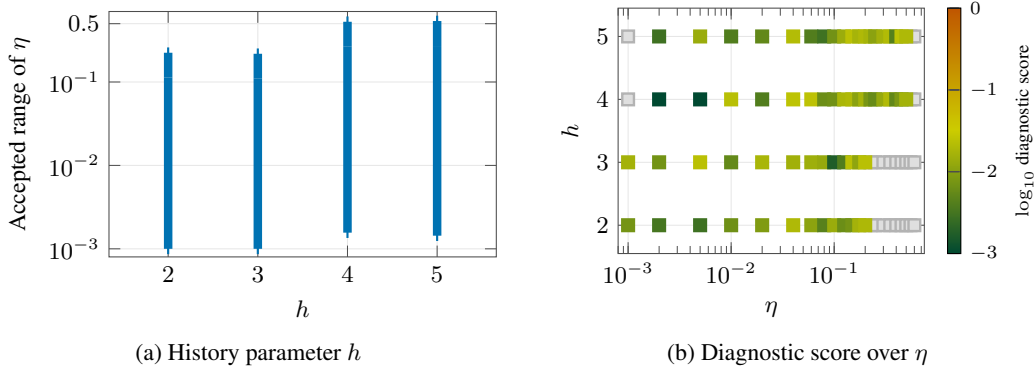

\section{Last-Iterate Behavior}\label{sec:pep}

In the previous sections, we established the global $O(1/T)$ convergence rate of AltGDA inspired by the PEP-based framework that seeks a quadratic Lyapunov certificate. A natural question is: \emph{can we achieve a last-iterate convergence rate for AltGDA by tuning its stepsize $\eta >0$}? Unfortunately, the answer is no: we give simple counterexamples showing that the last-iterate duality gap of AltGDA does not converge to zero both in matrix games and in bilinear problems over general compact convex sets. For AltGDA over compact convex sets, we provide a simple $2$-dimensional counterexample over the unit ball such that no matter what $\eta>0$ we use, the last-iterate duality gap remains a constant across all iterations. Second, for matrix games with $\mathcal X=\Delta_n$ and $\mathcal Y=\Delta_m$, we present one  \(2 \times 2\) matrix and initialization such that AltGDA has its last-iterate duality gap bounded away from zero for every $\eta >0$. Our explicit counterexamples are motivated again by a PEP-based framework \citep{drori2014performance,taylor2017smooth}  that numerically computes the worst-case last-iterate duality gap of AltGDA over compact convex sets.

Appendix~\ref{subsec:pep-exact-witnesses} proves the results associated with both counterexamples below.

\begin{example}[Last-iterate duality gap of AltGDA for bilinear problems over compact convex sets]
\label{prop:pep-exact-plateau}
Let \(\mathcal P_T(\eta)\) denote the worst-case last-iterate duality gap
of projected AltGDA over all dimensions, all compact convex sets
\(\mathcal X,\mathcal Y\) contained in Euclidean unit balls, all
\(A\) with \(\lVert A\rVert_2\leq1\), and all feasible initializations.
For every constant stepsize \(\eta>0\) and time horizon \(T\geq1\), we have
\[
\mathcal P_T(\eta)=2,
\text{ and } \inf_{\eta>0}\mathcal P_T(\eta)=2.
\]
\end{example}

The upper bound is attained on two-dimensional unit balls with \(A=I_2\).
The construction has gap \(2\) at every iterate, including initialization.

\begin{example}[Last-iterate duality gap of AltGDA for matrix games]
\label{ex:simplex-last-iterate}
On \(\Delta_2\times\Delta_2\), consider
\[
A=\begin{pmatrix}1&-1\\-1&1\end{pmatrix}, \text{ with }
x^0=(3/4,1/4), \text{ and } y^0=(1/2,1/2).
\]
The unique Nash equilibrium for this game is \(x^\star=y^\star=(1/2,1/2)\).
For this matrix game, for every fixed stepsize \(\eta>0\), AltGDA has its last-iterate duality gap bounded away from zero for every $\eta >0$, i.e., it satisfies $\inf_{t\geq0}\Gap(x^t,y^t)>0$. 

By our convergence results, for every fixed \(0<\eta\leq1/2\), while in last-iterate duality gap we have $\Gap(x^t,y^t)\geq 1/2$, AltGDA converges in ergodic duality gap at a rate $\Gap(\overline x^T,\overline y^T)\leq 5/(4\eta T)$.
Thus averaging yields convergence on a trajectory whose last-iterate
gap never vanishes. We provide the proof to this claim in \cref{rem:simplex-uniform-smallsteps}. 

\end{example}

Figure~\ref{fig:pep-lastiterate} validates
\cref{prop:pep-exact-plateau} numerically using the PEP framework. It shows the computed worst-case last-iterate duality gaps for $T = 5,\ldots,30$ are essentially constant at 2 for all practical purposes over a stepsize grid containing 25 points in $[1/64, 2]$. In Appendix~\ref{subsec:pep-protocol} we describe the PEP framework in detail.

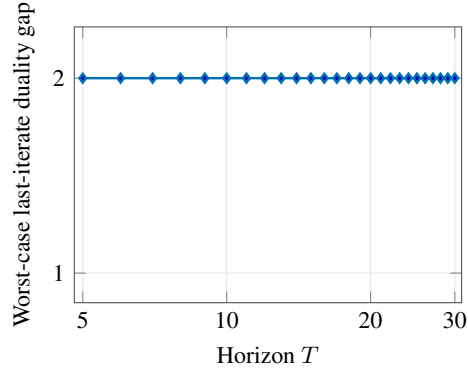
\begin{figure}[t]
    \centering
    \colorlet{seriesblue}{lmiFull}
    \begin{minipage}[t]{0.48\linewidth}
        \centering
        \begin{tikzpicture}
          \begin{axis}[
              width=\linewidth,
              height=0.78\linewidth,
              xmode=log,
              ymode=log,
              log basis x=10,
              log basis y=10,
              xmin=4.8, xmax=31,
              ymin=0.9, ymax=2.4,
              xtick={5,10,20,30},
              xticklabels={5,10,20,30},
              ytick={1,2},
              yticklabels={1,2},
              xlabel={Horizon $T$},
              ylabel={Worst-case last-iterate duality gap},
              xlabel style={font=\footnotesize},
              ylabel style={font=\footnotesize},
              tick label style={font=\footnotesize},
              grid=major,
              major grid style={gray!20},
              axis line style={black!70},
              clip marker paths=true,
            ]
            \addplot+[
              color=seriesblue,
              solid, mark=diamond*, mark size=1.7pt, line width=0.9pt,
            ] table[x=T, y=last_iterate_duality_gap, col sep=comma,
            ] {data/pep_ergodic_lastiterate_gap_figure_data.csv};
          \end{axis}
        \end{tikzpicture}
    \end{minipage}
    \caption{Numerical confirmation of the exact worst-case last-iterate gap
\(2\) over normalized compact convex sets
(\cref{prop:pep-exact-plateau}), for \(T=5,\dots,30\).}
    \label{fig:pep-lastiterate}
\end{figure}

\section{Conclusion}\label{sec:discussion}

In this paper we established a global \(O(1/T)\) ergodic convergence rate for
AltGDA in every finite two-player zero-sum matrix game, for every initialization and every horizon
$T$. Thus, alternation alone
suffices to obtain this rate in general matrix games.
We have formalized and machine-checked this global ergodic convergence
result in Lean~4. Our proof was guided by numerical observations from a PEP
framework for quadratic Lyapunov function search over compact
convex sets, which illustrates how systematic
numerical searches of Lyapunov functions can inform analytical convergence
proofs. Additionally, we provided simple counterexamples showing that the last-iterate duality gap of AltGDA does not converge to zero both in matrix games and in bilinear problems over general compact convex sets. This last-iterate behavior justifies why averaging of iterates is indeed necessary to achieve an \(O(1/T)\) rate.

\section*{Acknowledgments}

This research was supported by the Office of Naval Research awards
N00014-22-1-2530 and N00014-23-1-2374, and the National Science Foundation
awards IIS-2147361 and IIS-2238960. S. Das Gupta acknowledges support
from AFOSR Grant Number FA9550-25-1-0183.

\subsection*{AI use statement}

For developing our numerical PEP framework and deriving analytical convergence results in this paper,  we designed a \textit{domain-specific harness} using our domain expertise and used it with GPT 5.6 Sol. A  domain-specific harness is a user-written framework built around a general purpose LLM to enhance the LLM's domain specific capabilities. Our harness for the design and analysis of first-order methods for constrained bilinear games is inspired by the AutoOPT harness of \citet{kim2026autoopt} that is designed for an unconstrained single function setup. Our harness is composed of three agent \emph{skill}s; an agent skill is a collection of simple markdown files and scripts for extending agentic capabilities of an LLM, for more information see \url{https://agentskills.io/home}. We designed the first \emph{skill}  to assist us in modeling and solving the PEP problems arising in constrained bilinear games to assist us with the formulation, derivation,
implementation, and validation of the Lyapunov LMI PEP problem and the finite-horizon last-iterate duality-gap PEP problem in \cref{sec:lmi} and \cref{sec:pep}. The second skill of our harness helped us in developing analytical convergence arguments described in \cref{sec:main-conv-altgda} by observing and discovering analytical patterns obtained by solving the Lyapunov LMI PEP problem numerically. After manually verifying the correctness of the analytical convergence proofs, in the final step we used a third component of our harness to help us in formalizing and machine checking all these convergence results in Lean 4.32.0 with mathlib 4.32.0. For our Lean code, we ran two rounds of checks. In the first check, we used Lean to verify the ergodic convergence rate results with every step being explicit and the underlying logic verified, whereas in the second step we used \texttt{Comparator} independently (\url{https://github.com/leanprover/comparator}) to ensure that the Lean code proves exactly the statements it claims, does not use any forbidden axioms, and the submitted proofs are accepted by Lean kernel. 

We thoroughly reviewed any AI-assisted mathematical
arguments, code, computations, and prose and we accept responsibility for the correctness of the results in the manuscript.

\subsection*{Ethics statement}

Our work is theoretical and computational, and it uses no human subjects,
personal or sensitive data, user studies, or deployed decision systems.  Our paper studies convergence in bilinear saddle point problems and we make no deployment
recommendation either.

\subsection*{Reproducibility statement}
The Lean verification code and Julia experiment code are provided at:
\begin{center}
\url{https://github.com/Shuvomoy/AltGDA-code}
\end{center}
All the convergence results in
Sections~\ref{sec:setting} and \ref{sec:architecture} and in
Appendix~\ref{sec:appendix-proof} of the paper are formalized and
machine-checked using Lean 4.32.0 with mathlib 4.32.0. The numerical PEP
frameworks for the Lyapunov-function search in \cref{sec:lmi} and
Appendix~\ref{sec:appendix-lmi} and for the finite-horizon PEP study for
investigating last-iterate behavior of AltGDA can be reproduced using Julia.

\bibliography{refs}

\bibliographystyle{iclr2027_conference}

\appendix
\counterwithin{equation}{section}
\counterwithin{theorem}{section}
\counterwithin{lemma}{section}
\counterwithin{remark}{section}
\counterwithin{example}{section}

\appendix

\begin{center}
    {\Large Appendices}
\end{center}

\section{Full Proof of the Main Theorem}\label{sec:appendix-proof}
\subsection{Overview, setting, and notation}
\label{sec:overview}

This appendix gives a self-contained proof of \cref{thm:main-text}.
It first restates the setting of~\cref{sec:setting} 
and then constructs the analysis witness (we call the fixed Goldman-Tucker strictly complementary saddle point a \emph{witness} in the proof) needed to state the result again as~\cref{thm:main}. 
The roadmap at the end of this subsection aligns the detailed proof with the proof sketch in~\cref{sec:architecture}.

Consider a finite two-player zero-sum game with payoff matrix \( A\in\R^{m\times n} \): 
one player mixes over \(n\) actions to minimize the expected payoff and the other mixes over \(m\) actions to maximize it. 
The simplest learning dynamics for such a game is letting both players run projected gradient steps simultaneously. 
We study the \emph{alternating} variant, in which the maximizer moves second, against the minimizer's freshly updated strategy. 
The averages of the players' strategies then approach an equilibrium at the rate \(O(1/T)\) for any $T \geq 1$, 
from every initialization, 
with any appropriate stepsize.

Let
\[
    \Delta_d
    =
    \left\{z\in\R^d:\ z\geq0,\ \one^\top z=1\right\}
\]
denote the probability simplex in \(\R^d\), where \(\R^d\) is the $d$ dimensional Euclidean space. 
The matrix game can be represented by 
\begin{equation}
    \min_{x\in\Delta_n} \max_{y\in\Delta_m} y^\top Ax,
    \label{eq:game}
\end{equation}
where the \(x\)-player minimizes and the \(y\)-player maximizes the payoff. 
By the minimax theorem~\citep{vonNeumann1928}, 
the game has a value $\nu$ and a saddle point. 
The duality gap measures how near a candidate pair $( \widetilde{x}, \widetilde{y}) \in \Delta_n \times \Delta_m$ is to being a saddle point:
\begin{align}
    \Gap(\widetilde x,\widetilde y)
    &=
    \max_{y\in\Delta_m}y^\top A\widetilde x
    -
    \min_{x\in\Delta_n}\widetilde y^\top Ax
    \notag\\
    &=
    \max_{\substack{x\in\Delta_n, y\in\Delta_m}}
    \left(
      y^\top A\widetilde x-\widetilde y^\top Ax
    \right),
    \label{eq:gap}
\end{align}
which is nonnegative everywhere and zero exactly at saddle points.

We study alternating projected gradient descent--ascent (AltGDA). 
With a fixed stepsize $\eta > 0$ and \(\Pi_C\) denoting Euclidean projection onto a closed convex set \(C\), 
AltGDA generates
\begin{equation}
    x^{t+1}
    =
    \Pi_{\Delta_n}\!\left(x^t-\eta A^\top y^t\right),
    \qquad
    y^{t+1}
    =
    \Pi_{\Delta_m}\!\left(y^t+\eta Ax^{t+1}\right),
    \qquad \forall\, t \geq 0.
    \label{eq:altgda}
\end{equation}
The alternation trick is applied as we use \(x^{t+1}\) to update $y^{t+1}$ instead of using \(x^t\). 
We track the uniform averages of the players' strategies: 
\begin{equation}
    \overline x^T=\frac1T\sum\nolimits_{t=1}^T x^t,
    \qquad
    \overline y^T=\frac1T\sum\nolimits_{t=1}^T y^t.
    \label{eq:averages}
\end{equation}

\paragraph{The main result.}
It is known that every matrix game possesses a strictly complementary saddle point, by the Goldman--Tucker theorem. 
From one such saddle point,
we can extract a scalar parameter $\delta \in (0,1]$, 
called here the \emph{Goldman--Tucker separation constant} (see~\ref{sec:witness}): 
roughly speaking, this parameter captures 
the smallest equilibrium probability on the supports (of the selected saddle point) and the smallest normalized off-support payoff slack.

Then, \cref{thm:main} states that, for any $L > 0$ with \(\lVert A\rVert_2 \leq L\), 
AltGDA with every fixed stepsize \( 0 < \eta \leq \delta/(2\sqrt2\,L) \) guarantees that 
\[
    \Gap\left( \overline{x}^T, \overline{y}^T\right) \leq \frac{15}{2\eta T}
    \quad \text{for every } T \geq 1
    \text{ and every initialization $(x^0, y^0)$.}
\]

\paragraph{The main difficulty.}
The standard route to an \(O(1/T)\) ergodic rate is an energy function argument: 
by exhibiting a nonnegative energy whose per-step decrease dominates the residual quantity to be summed, 
we can show that the residual terms in the regret analysis are bounded by a constant that is independent of $T$ no matter how large the horizon $T$ is, which leads to an $O(1/T)$ ergodic rate. 

This is how the $O(1/T)$ convergence rate is proven under the assumption of the existence of an interior NE~\citep{nan2026convergence}.
However, in the general matrix games without an interior NE, the difference of the usual energy function no longer yields an upper bound for the residual terms. 

In our proof, we identify that each AltGDA step obeys an \emph{exact} energy balance identity (Lemma~\ref{lem:energy}),
\begin{equation*}
    V_{t+1}-V_t+\mathscr D_t=r_t,
\end{equation*}
where \(\mathscr D_t \geq 0\) is the ``dissipation'' term, 
and \(r_t \geq 0\) is the residual term that needs to be controlled. 
The residual term corresponds to a collision term supported on deactivated coordinates. 
It is positive exactly when at least one such coordinate has both positive previous mass and a positive projection multiplier (Lemma~\ref{lem:residual}).
The proof then shows that the residual terms, summed over an arbitrary
horizon $T$, cost only a constant.

\paragraph{Organization.}
\Cref{sec:projection} first gives the KKT representation of the two projections and identifies the residual \(r_t\) exactly.
\cref{sec:energy} then proves the energy balance for an
alternation-corrected energy \(V_t\). 
\Cref{sec:reflection} establishes an important inequality: the residual term is upper bounded by a fraction of the dissipation term when we choose an appropriate stepsize.
With these tools in hand, we build a finite bound for the residual budget, equivalently, the summability of the residual term in 
\cref{sec:budget}. 
The final step uses 
the projection inequalities to convert that budget into the duality gap bound in \cref{sec:gap}.  

\paragraph{Notation.}
The vector \(\one\) has every entry one, with dimension determined by
context; \(I_d\) is the \(d\times d\) identity; \(e_i\) is the
\(i\)th standard basis vector.  Vector inequalities are
coordinatewise, and \(\odot\) is the coordinatewise product.  For an
index set \(S\), we write \(S^c\) for its complement, \(|S|\) for its
cardinality, \(\one_S\) for its indicator vector, and \(g_S\) for the
corresponding subvector of \(g\); \(\supp(z)\) is the set of indices
of nonzero coordinates of \(z\).  The norm \(\lVert\cdot\rVert_2\) is
Euclidean on vectors and spectral on matrices, and for a symmetric
positive semidefinite matrix \(G\) we write
\(\lVert z\rVert_G^2=z^\top Gz\).

\subsection{The Goldman--Tucker saddle point and the main theorem}
\label{sec:witness}

The proof uses one strictly complementary saddle point at which supports
and payoff slacks separate cleanly. Writing $\nu$ for the game value,
the $x$-player's payoff slack on coordinate $i$ is
$(A^\top y^\star)_i-\nu$.
At this selected saddle point, on-support coordinates ($x^\star_i>0$)
have zero payoff slack, i.e., $(A^\top y^\star)_i=\nu$,
whereas off-support coordinates ($x^\star_i=0$) have strictly positive
payoff slack, i.e., $(A^\top y^\star)_i>\nu$.

The existence of such a saddle point can be shown by writing the matrix games as linear programs. 
The bilinear two player zero sum games, written as linear programs in the strategy and a value variable, are
\begin{align}
    \min_{x,\theta}\quad &\theta
    &
    \text{subject to}\quad&
    Ax\leq\theta\one,\quad
    \one^\top x=1,\quad x\geq0,
    \label{eq:primal-lp}\\
    \max_{y,\omega}\quad &\omega
    &
    \text{subject to}\quad&
    A^\top y\geq\omega\one,\quad
    \one^\top y=1,\quad y\geq0.
    \label{eq:dual-lp}
\end{align}
These constitute a primal-dual linear programming pair, 
and the strict-complementarity theorem for general-form programs~\citep{Williams1970} applies; 
equivalently, the Goldman-Tucker theorem~\citep{GoldmanTucker1956} applies after a routine conversion of the free value variables to standard form.  
Either theorem yields an optimal pair $(x^\star, y^\star)$ with common value $\nu$ whose slack vectors
\begin{equation}
    s_x=A^\top y^\star-\nu\one,\qquad
    s_y=\nu\one-Ax^\star
    \label{eq:slacks}
\end{equation}
are nonnegative and satisfy, coordinatewise,
\begin{equation}
    x_i^\star s_{x,i}=0,\qquad x_i^\star+s_{x,i}>0,
    \qquad
    y_j^\star s_{y,j}=0,\qquad y_j^\star+s_{y,j}>0.
    \label{eq:strict-complementarity}
\end{equation}
In words, on each coordinate, either the equilibrium strategy places
mass or the payoff slack is positive, and never both. 
In our proof of the main theorem, 
we first fix one such pair and call it the \emph{Goldman-Tucker saddle point}. 
In the remainder of this appendix, we use $(x^\star, y^\star)$ to denote the selected Goldman-Tucker saddle point. 

This saddle point determines the index sets (on-support and off-support sets of coordinates) and 
separation constants of the analysis. 
Let $I = \supp(x^\star)$ and \(J=\supp(y^\star)\), 
with counts \(k_x=|I|\), \(\ell_x=|I^c|\), \(k_y=|J|\), and
\(\ell_y=|J^c|\).
Note that the support sizes \(k_x,k_y\) are nonzero. 
We define the smallest on-support probabilities: 
\[
    \underline x=\min_{i\in I}x_i^\star,
    \qquad
    \underline y=\min_{j\in J}y_j^\star,
\]
and the smallest off-support payoff slacks: 
\[
    \sigma_x=\min_{i\in I^c}s_{x,i},
    \qquad
    \sigma_y=\min_{j\in J^c}s_{y,j},
\]
where \(\sigma_x\) is defined when \(I^c\neq\varnothing\) and \(\sigma_y\) is defined when \(J^c\neq\varnothing\). 
Thus, \(\underline x,\underline y>0\), 
and every defined payoff slack minimum is positive.
Let $L > 0$ satisfy \(\lVert A\rVert_2\leq L\), and set
\begin{equation}
    \delta
    =
    \min\left\{
      \underline x,\ \underline y,\
      \frac{\sigma_x}{L}\ \text{if }I^c\neq\varnothing,\
      \frac{\sigma_y}{L}\ \text{if }J^c\neq\varnothing
    \right\},
    \label{eq:delta}
\end{equation}
where terms corresponding to empty off-support index sets (i.e., $| I | = n$ or $| J | = m$) are omitted.
Then, we have \(\delta>0\) and \(\delta\leq1\) because \(\underline x\leq1\) and $\underline{y} \leq 1$.
We call this scalar the \emph{Goldman-Tucker separation constant}
associated with the selected Goldman-Tucker saddle point and the chosen bound \(L\).
Informally, it records how decisively the Goldman-Tucker saddle point separates its supports from their off-support coordinates, which can be measured by the minimum of \emph{the least probability assigned to an on-support action} and \emph{the least normalized payoff slack of an off-support action}.

The following theorem restates \cref{thm:main-text}.
\begin{theorem}
    \label{thm:main}
    Let \(A\in\R^{m\times n}\), let \(L>0\) satisfy
    \(\lVert A\rVert_2\leq L\), and define \(\delta\) from a selected
    Goldman--Tucker strictly complementary saddle point as in
    \eqref{eq:delta}.  For every initialization
    \((x^0,y^0)\in\Delta_n\times\Delta_m\), every \(T\geq1\), and every
    fixed stepsize satisfying
    \begin{equation}
        0<\eta\leq\frac{\delta}{2\sqrt{2}\,L},
        \label{eq:stepsize}
    \end{equation}
    the averages \eqref{eq:averages} of the AltGDA iterates
    \eqref{eq:altgda} satisfy, in the duality gap \eqref{eq:gap},
    \begin{equation}
        \boxed{
        \Gap(\overline x^T,\overline y^T)
        \leq\frac{15}{2\eta T}.}
        \label{eq:main-rate}
    \end{equation}
\end{theorem}

In the subsequent sections, we assume a nonzero $A \neq 0$ (otherwise the matrix game is trivial). Note that we can choose an arbitrary \(L>0\) with
\(\lVert A\rVert_2\leq L\), and in practice this can be simply done by setting $L = \lVert A \rVert_2$.

\subsection{Projection and residual identities}
\label{sec:projection}

The first-order description of the two projections provides upper bounds for the following ``per-step iterate movement'': 
\begin{equation}
    \Delta x_t=x^{t+1}-x^t,\qquad
    \Delta y_t=y^{t+1}-y^t,
    \qquad
    v_t=-A^\top y^t,\qquad
    u_t=Ax^{t+1}
    \label{eq:increments}
\end{equation}
Here, $v_t$ and $u_t$ are shorthand notations for the signed payoff vectors that drive the movement of the iterates.
We use \(\Delta\) to denote the forward difference at some time
index: 
for any time-indexed quantity $a_t$, we write $\Delta a_t=a_{t+1}-a_t$.

The first-order description also provides a classic and useful residual representation, which we call the projection KKT representation of the AltGDA algorithm.
\begin{lemma}[Projection KKT representation of AltGDA]
    \label{lem:kkt}
    For every \(t \geq 0\), 
    there are scalars \(\gamma_t, \lambda_t\) and vectors \(\mu_t\in\R_+^n,\rho_t\in\R_+^m\) such that
    \begin{align}
        \Delta x_t
        &=
        \eta(v_t-\gamma_t\one+\mu_t),
        &
        \mu_t\odot x^{t+1}&=0,
        \label{eq:kkt-x}\\
        \Delta y_t
        &=
        \eta(u_t-\lambda_t\one+\rho_t),
        &
        \rho_t\odot y^{t+1}&=0.
        \label{eq:kkt-y}
    \end{align}
    Moreover, we have 
    \begin{equation}
        \lVert\Delta x_t\rVert_2\leq\eta L,
        \qquad
        \lVert\Delta y_t\rVert_2\leq\eta L.
        \label{eq:movement}
    \end{equation}
\end{lemma}

\begin{proof}
By the KKT conditions for the Euclidean projection
\(p^+=\Pi_{\Delta_d}(p+\eta g)\), we have 
\begin{equation}
    p^+-p=\eta(g-\gamma\one+\mu),
    \qquad
    \mu\geq0,\qquad
    \mu\odot p^+=0,
    \label{eq:projection-KKT-conditions}
\end{equation}
where \(\gamma\) is the (scalar) multiplier of the sum constraint and 
\(\mu\) is a vector corresponding to the multipliers of the sign constraints. 
Applying~\cref{eq:projection-KKT-conditions} with \((p,g)=(x^t,v_t)\) and \((p,g)=(y^t,u_t)\), we have~\cref{eq:kkt-x,eq:kkt-y}.

For the per-step iterate movement bound, we use the nonexpansive property of the projections and the definition of the induced $2$-norm: 
\[
    \lVert\Delta x_t\rVert_2
    =
    \bigl\lVert
      \Pi_{\Delta_n}(x^t-\eta A^\top y^t)-\Pi_{\Delta_n}(x^t)
    \bigr\rVert_2
    \leq\eta\lVert A^\top y^t\rVert_2
    \leq\eta L\lVert y^t\rVert_2
    \leq\eta L,
\]
where the last step uses
\(\lVert y^t\rVert_2\leq\one^\top y^t=1\).  The \(y\)-iterate movement bound follows identically.
\end{proof}

In~\cref{eq:kkt-x,eq:kkt-y}, the multiplier \(\mu_t\) is strictly positive only on those coordinates on which the next-step iterate touches the simplex boundary
($x_i^{t+1} = 0$), and likewise for \(\rho_t\). 

Next, we introduce a quantity called ``residual term'' in the proof: 
\begin{equation}
    r_t :=
    \eta v_t^\top\Delta x_t
    +\eta u_t^\top\Delta y_t
    -\lVert\Delta x_t\rVert_2^2
    -\lVert\Delta y_t\rVert_2^2.
    \label{eq:residual-definition}
\end{equation}
This is a key quantity in the convergence proof for AltGDA.
In the unconstrained case, each iterate movement equals \(\eta\) times the corresponding payoff vector and hence \(r_t\) vanishes. 
In the constrained case, it is much trickier to handle. 
The next lemma computes a more useful equivalent closed-form expression of \(r_t\).
Since $\eta>0$ and the probabilities and projection multipliers are
nonnegative, \cref{eq:residual-kkt} implies that $r_t>0$ if and only if
there is a coordinate $i$ with $x_i^t>0$ and $\mu_{t,i}>0$, or a
coordinate $j$ with $y_j^t>0$ and $\rho_{t,j}>0$.
By complementarity in~\cref{lem:kkt}, these conditions imply
$x_i^{t+1}=0$ or $y_j^{t+1}=0$, respectively.
Thus, the residual is strictly positive exactly when at least one player
has a previously positive coordinate reach the simplex boundary with a
strictly positive projection multiplier; merely touching the boundary
is not sufficient.
\begin{lemma}[Exact residual representation]
\label{lem:residual}
For every \(t\geq0\),
\begin{equation}
    r_t
    =
    \eta\left(
      \mu_t^\top x^t+\rho_t^\top y^t
    \right)
    \geq0.
    \label{eq:residual-kkt}
\end{equation}
\end{lemma}

\begin{proof}
Since \(\one^\top\Delta x_t=0\), by~\cref{eq:kkt-x} we have 
\begin{align*}
    \eta v_t^\top\Delta x_t-\lVert\Delta x_t\rVert_2^2
    &=
    \eta\left(v_t-\frac{\Delta x_t}{\eta}\right)^\top
    \Delta x_t\\
    &=
    \eta(\gamma_t\one-\mu_t)^\top\Delta x_t\\
    &=
    -\eta\mu_t^\top(x^{t+1}-x^t)
    =
    \eta\mu_t^\top x^t,
\end{align*}
where the last equality follows because \(\mu_t^\top x^{t+1}=0\). 
By a symmetric proof, we can show that the \(y\)-player's identity holds as well. 
Furthermore, $r_t$ is nonnegative for any $t \geq 0$ because every factor in~\cref{eq:residual-kkt} is nonnegative.
\end{proof}

\subsection{An energy function for AltGDA}
\label{sec:energy}

Another important component in the proof is an energy function. 
A natural choice of the energy function for the minimax problems is the squared distance to a saddle point. 
To account for the asymmetry introduced by alternation, 
we add a bilinear cross term of size \(\eta\). 
With the Goldman-Tucker saddle point \((x^\star,y^\star)\) fixed,
define
\begin{equation}
    V_t
    =
    \lVert x^t-x^\star\rVert_2^2
    +\lVert y^t-y^\star\rVert_2^2
    -\eta(y^t-y^\star)^\top A(x^t-x^\star).
    \label{eq:energy}
\end{equation}
Alongside the energy, 
the analysis tracks two nonnegative quantities built from the current iterates and the Goldman-Tucker saddle point: 
the payoff slack-weighted probabilities
\[
    P_t^x=s_x^\top x^t,\qquad
    P_t^y=s_y^\top y^t,\qquad
    P_t=P_t^x+P_t^y,
\]
and the equilibrium probability-weighted multipliers 
\[
    E_t^x=\mu_t^\top x^\star,\qquad
    E_t^y=\rho_t^\top y^\star,\qquad
    E_t=E_t^x+E_t^y.
\]
In particular, taking the $x$-player side as an example, 
because \(s_x\) vanishes on \(I\), 
the quantity \(P_t^x\) measures the off-support iterate probabilities, with each coordinate weighted by its payoff slack.
Similarly, because \(x^\star\) vanishes outside \(I\), 
the quantity \(E_t^x\) measures the multipliers on the support, 
with each coordinate weighted by its equilibrium probability. 

We then define a quantity that combines $P_t$ and $E_t$: 
\begin{equation}
    \mathscr D_t
    =
    \eta(P_t+P_{t+1})+2\eta E_t. 
    \label{eq:dissipation}
\end{equation}
We call the above quantity the \emph{dissipation} of the $t$-th step, 
and all of these three quantities ($P_t$, $E_t$, and $\mathscr{D}_t$) are nonnegative by definition. 

The central identity of the analysis is that energy, dissipation, and residual are in exact balance.
\begin{lemma}[Exact energy identity]
\label{lem:energy}
For every \(t\geq0\),
\begin{equation}
    \boxed{V_{t+1}-V_t+\mathscr D_t=r_t.}
    \label{eq:energy-balance}
\end{equation}
\end{lemma}

\begin{proof}
Set
\[
    a=x^t-x^\star,\qquad
    b=y^t-y^\star,\qquad
    p=\Delta x_t,\qquad
    h=\Delta y_t.
\]
Because the sums over all coordinates of both $a$ and $b$ are zero, 
the KKT conditions~\cref{eq:kkt-x,eq:kkt-y} and the definitions of the payoff slack in~\cref{eq:slacks} imply 
\begin{align}
    a^\top p
    &=
    -\eta P_t^x-\eta b^\top Aa+\eta\mu_t^\top a,
    \label{eq:energy-ap}\\
    b^\top h
    &=
    -\eta P_t^y+\eta b^\top Aa+\eta b^\top Ap
    +\eta\rho_t^\top b.
    \label{eq:energy-bh}
\end{align}
Indeed, \(a^\top A^\top y^\star=a^\top s_x=P_t^x\) because
\(\one^\top a=0\) and \(s_x^\top x^\star=0\), 
and \(b^\top Ax^\star=-b^\top s_y=-P_t^y\) by the same arguments.

Expanding the two squared distances and the bilinear term in the definition of $V_t$ in~\cref{eq:energy}, 
and then using \cref{eq:energy-ap,eq:energy-bh}, 
we have 
\begin{align}
    V_{t+1}-V_t
    ={}&
    -2\eta P_t
    +2\eta(\mu_t^\top a+\rho_t^\top b)
    +\lVert p\rVert_2^2+\lVert h\rVert_2^2
    \notag\\
    &+\eta b^\top Ap-\eta h^\top A(a+p).
    \label{eq:energy-expanded}
\end{align}
Furthermore, it holds that 
\[
    P_{t+1}=P_t+y^{\star\top}Ap-h^\top Ax^\star,
\]
and, by the definition of $r_t$ in~\cref{eq:residual-definition},
\[
    -r_t
    =
    \eta y^{t\top}Ap-\eta h^\top Ax^{t+1}
    +\lVert p\rVert_2^2+\lVert h\rVert_2^2.
\]
Then, substituting \(y^t=y^\star+b\) and \(x^{t+1}=x^\star+a+p\) into these two expressions transforms \eqref{eq:energy-expanded} into
\begin{equation}
    V_{t+1}-V_t
    =
    -\eta(P_t+P_{t+1})
    +2\eta\left[
      \mu_t^\top(x^t-x^\star)
      +\rho_t^\top(y^t-y^\star)
    \right]-r_t.
    \label{eq:difference-of-energy}
\end{equation}
By Lemma~\ref{lem:residual} and the definition of $E_t$, the terms in the bracket equal \(r_t/\eta-E_t\).
By substituting \(r_t/\eta-E_t\) into~\cref{eq:difference-of-energy}, we prove~\cref{eq:energy-balance}.
\end{proof}

We also present two elementary bounds on the energy function.
These bounds show the energy function is bounded by some finite constants and are useful to conclude the main theorem.

We use a shorthand notation 
\[
    q=\eta L 
\]
to denote the per-step iterate movement scale (see~\cref{lem:kkt}). 
Since \(\delta\leq1\) and our selected stepsize satisfies~\cref{eq:stepsize}, we have:
\begin{equation}
    q\leq\frac1{2\sqrt2}.
    \label{eq:q-bound}
\end{equation}
The bilinear cross term in the definition of $V_t$ in~\cref{eq:energy} cannot overwhelm the squared distances: 
by Cauchy-Schwarz inequality and \(2ab \leq a^2+b^2\) for any scalars $a, b$, we have 
\begin{equation}
    V_t
    \geq
    \left(1-\frac q2\right)
    \left(
      \lVert x^t-x^\star\rVert_2^2
      +\lVert y^t-y^\star\rVert_2^2
    \right)
    \geq0.
    \label{eq:energy-lower}
\end{equation}
On the other hand, since the squared diameter of a simplex is at most \(2\), we can also upper bound the initial energy: 
\begin{equation}
    V_0\leq4+2q.
    \label{eq:energy-initial}
\end{equation}

Summing~\cref{eq:energy-balance} shows that the total dissipation equals the total residual plus two extra bounded energy terms. 
If each residual is bounded by a fraction of the dissipation
plus a telescoping term, 
then we can show that the total residual is bounded by a constant. 
The next section establishes this relation.

\subsection{residual is upper bounded by a fraction of dissipation}
\label{sec:reflection}

This section establishes an important relation between $r_t$ and $\mathscr{D}_t$
\begin{equation*}
    r_t\leq B_t-B_{t+1}+\tfrac12\mathscr D_t, 
\end{equation*}
where \(B_t\) is a uniformly bounded term. 
By~\cref{lem:residual}, the residual splits into the two player
blocks \(\eta\mu_t^\top x^t\) and \(\eta\rho_t^\top y^t\).
We next present the \(x\)-player's block in detail, and then show the results corresponding to $y$-player's block, which follow by the same argument as the \(x\)-player's block.

We first define some additional notations and an auxiliary lemma.

For a nonempty index set \(S\) and a vector \(g\), define 
\begin{equation}
    (\mathsf C_Sg)_i
    =
    \frac1{|S|}\sum_{j\in S}g_j-g_i,
    \qquad i\in S^c.
    \label{eq:contrast}
\end{equation}

\begin{lemma} 
    \label{lem:contrast}
    Let \(S\) have cardinality \(k\).  For \(i\in S^c\),
    \begin{equation}
        |(\mathsf C_Sg)_i|
        \leq
        \sqrt{1+\frac1k}\,\lVert g\rVert_2.
        \label{eq:contrast-norm}
    \end{equation}
    If \(h\in\R^k\) and \(\overline h=k^{-1}\one^\top h\), then for every
    \(i\in S\),
    \begin{equation}
        |\overline h-h_i|
        \leq
        \sqrt{1-\frac1k}\,\lVert h\rVert_2.
        \label{eq:centering-norm}
    \end{equation}
\end{lemma}

\begin{proof}
For \(i\notin S\),
\[
    (\mathsf C_Sg)_i
    =
    \ip{k^{-1}\one_S-e_i}{g},
\]
and the coefficient vector $k^{-1}\one_S-e_i, \, i \in S^c$ has squared norm \(1/k+1\). 
For
\(i\in S\),
\[
    \overline h-h_i
    =
    \ip{k^{-1}\one_S-e_i}{h},
\]
where the coefficient vector $k^{-1}\one_S-e_i, \, i \in S$ has squared norm $1/k - 2/k + 1 = 1 - 1/k$, because the entry at \(i\) is \(1/k-1\).  Combining both inner products with Cauchy--Schwarz inequality yields~\cref{eq:contrast-norm,eq:centering-norm}.
\end{proof}

\subsubsection{Upper bound the \texorpdfstring{\(x\)}{x}-player block}
\label{app:subsec:x-block}

Assume first that \(I^c\neq\varnothing\); the case \(I^c=\varnothing\) is easier and is handled at the end of this subsection. 
We introduce a set of notations corresponding to the off-support coordinates: 
\[
    O_x=I^c,\qquad
    z_t=x^t_{O_x},\qquad
    M_t^x=\one^\top z_t.
\]
In words, \(z_t\) is the restriction of the iterate $x^t$ to the off-support coordinates and \(M_t^x\) is its total probability over all off-support coordinates.
Furthermore, we define
\[
    d_t^x=\mathsf C_I(-A^\top y^t),\qquad
    e_t^x=\frac1{k_x}\sum_{i\in I}\mu_{t,i},\qquad
    G_x=I_{\ell_x}+\frac1{k_x}\one\one^\top.
\]
Here, \(d_t^x\) is the contrast of the payoff vectors on the off-support coordinates against the on-support average of the payoff vectors, 
\(e_t^x\) is the average of the multipliers over the on-support coordinates, and \(G_x\succeq I\) is a positive semidefinite matrix.

The following lemma provides a useful decomposition of the off-support component $\eta(\mu_t)_{O_x}^\top z_t$.
\begin{lemma} 
\label{lem:reflection-x}
For every \(t\geq0\),
\begin{align}
    G_x\Delta z_t
    &=
    -\eta d_t^x-\eta e_t^x\one+\eta(\mu_t)_{O_x},
    \label{eq:reflection-vector-x}\\
    \eta(\mu_t)_{O_x}^\top z_t
    &=
    -\lVert\Delta z_t\rVert_{G_x}^2
    -\eta(d_t^x)^\top\Delta z_t
    -\eta e_t^x\Delta M_t^x.
    \label{eq:reflection-scalar-x}
\end{align}
\end{lemma}

\begin{proof}
By summing~\cref{eq:kkt-x} over \(I\), using
\(\one^\top\Delta x_t=0\), and rearranging terms, we have 
\[
    \gamma_t
    =
    \frac1{k_x}\sum_{i\in I}v_{t,i}
    +e_t^x+\frac{\Delta M_t^x}{\eta k_x}.
\]
Substituting the above expression into the KKT representations in~\cref{lem:kkt} proves~\cref{eq:reflection-vector-x}. 
By taking the inner product of~\cref{eq:reflection-vector-x} with \(\Delta z_t\), and using \((\mu_t)_{O_x}^\top z_{t+1}=0\), we obtain~\cref{eq:reflection-scalar-x}.
\end{proof}
Set \(d_{-1}^x=d_0^x\) and define the
storage
\begin{equation}
    B_{x,t}=\eta(d_{t-1}^x)^\top z_t.
    \label{eq:storage-x}
\end{equation}
Then, we have 
\begin{equation}
    -\eta(d_t^x)^\top\Delta z_t
    =
    B_{x,t}-B_{x,t+1}
    +\eta(d_t^x-d_{t-1}^x)^\top z_t.
    \label{eq:storage-telescope-x}
\end{equation}
For \(t\geq1\), linearity of $C_I$ gives
\[
    d_t^x-d_{t-1}^x
    =
    \mathsf C_I(-A^\top\Delta y_{t-1}).
\]
This difference is zero at \(t=0\) by the convention above.
Combining the iterate movement bound~\cref{eq:movement} with 
the bound~\cref{eq:contrast-norm} on $C_I$ yields
\begin{equation}
    \lVert d_t^x-d_{t-1}^x\rVert_\infty
    \leq
    \sqrt{1+\frac1{k_x}}\,
    \eta L^2.
    \label{eq:drift-x}
\end{equation}
We then connect the $\ell_1$ norm of $z_t$ to the dissipation term $P^x_t$ by leveraging the definition of $\sigma_x$: 
since \(s_{x,i}\geq\sigma_x\) for \(i\in O_x\) and \(s_x\) vanishes on
\(I\), strict complementarity gives us that 
\begin{equation}
    M_t^x\leq\frac{P_t^x}{\sigma_x}.
    \label{eq:outside-mass-x}
\end{equation}

Combining~\cref{eq:drift-x,eq:outside-mass-x} with~\cref{eq:storage-telescope-x}, we obtain an upper bound for $-\eta(d_t^x)^\top\Delta z_t$.

Next, we turn to the on-support component. 
Let
\[
    h_t^x=(\Delta x_t)_I,
    \qquad
    \overline h_t^x=\frac1{k_x}\one^\top h_t^x
\]
denote the on-support iterate movement and its mean. 
Since $\sum_{i \in I}(\Delta x_t)_i + \sum_{i \in I^c}(\Delta x_t)_i = 0$ and $(\mu_t)_I^\top x_I^{t+1} = 0$, we have 
\begin{equation}
    \Delta M_t^x=-\one^\top h_t^x,
    \qquad
    (\mu_t)_I^\top x_I^t=-(\mu_t)_I^\top h_t^x.
    \label{eq:mass-complementarity-x}
\end{equation}
Combining them yields the exact identity
\begin{equation}
    (\mu_t)_I^\top x_I^t-e_t^x\Delta M_t^x
    =
    (\mu_t)_I^\top
    \left(\overline h_t^x\one-h_t^x\right).
    \label{eq:cancellation-x}
\end{equation}
By the nonnegativity of \((\mu_t)_I\geq0\), 
the per-step iterate movement bound~\cref{eq:movement}, 
the bound~\cref{eq:centering-norm} for $(\overline{h}_t^x\one - h_t^x)$, and the lower bound
\(E_t^x=\mu_t^\top x^\star\geq\underline x\sum_{i\in I}\mu_{t,i}\), 
we have 
\begin{align}
    (\mu_t)_I^\top
    \left(\overline h_t^x\one-h_t^x\right)
    &\leq
    \sqrt{1-\frac1{k_x}}\,
    \lVert h_t^x\rVert_2
    \sum_{i\in I}\mu_{t,i}
    \notag\\
    &\leq
    \sqrt{1-\frac1{k_x}}\,
    \eta L\,\frac{E_t^x}{\underline x}.
    \label{eq:cancellation-bound-x}
\end{align}

Finally, we combine both off-support (\cref{eq:reflection-scalar-x}) and on-support (\cref{eq:cancellation-x}) components. 
Note that the term $\eta e_t^x\Delta M_t^x$ in~\cref{eq:reflection-scalar-x,eq:cancellation-x} can be exactly canceled, and $-\lVert\Delta z_t\rVert_{G_x}^2 \leq 0$ because of the positive semidefiniteness. 
By using the equations and inequalities mentioned above, in particular, \cref{eq:storage-telescope-x,eq:drift-x,eq:outside-mass-x,eq:cancellation-bound-x}, we attain an upper bound for the $x$-player block of $r_t$:
\begin{equation}
    \boxed{
    \begin{aligned}
    \eta\mu_t^\top x^t
    \leq{}&
    B_{x,t}-B_{x,t+1}\\
    &+
    \sqrt{1+\frac1{k_x}}\,
    \frac{\eta L^2}{\sigma_x}(\eta P_t^x)
    +
    \sqrt{1-\frac1{k_x}}\,
    \frac{\eta L}{\underline x}(\eta E_t^x).
    \end{aligned}}
    \label{eq:block-bound-x}
\end{equation}

At the end of this part, we consider \(I^c=\varnothing\).
In this case, we retain the definitions
\[
    e_t^x=\frac1{k_x}\sum_{i\in I}\mu_{t,i},
    \qquad
    h_t^x=(\Delta x_t)_I,
    \qquad
    \overline h_t^x=\frac1{k_x}\one^\top h_t^x,
\]
and set
\[
    M_t^x=\Delta M_t^x=B_{x,t}=0.
\]

In this case, only the last term in~\cref{eq:block-bound-x} is nonzero, and one can verify that~\cref{eq:block-bound-x} still holds by the same arguments as those showing~\cref{eq:cancellation-bound-x}.

\subsubsection{Upper bound the \texorpdfstring{\(y\)}{y}-player block}

The proof for the \(y\)-player block follows the same route as that for the $x$-player block.
Because the \(y\)-player's update at the \(t\)-th iteration is driven by
\(Ax^{t+1}\), 
there is a time-index change across the arguments.
Again, we first assume \(J^c\neq\varnothing\), 
and define
\[
    O_y=J^c,\qquad
    w_t=y^t_{O_y},\qquad
    M_t^y=\one^\top w_t,
\]
and
\[
    d_t^y=\mathsf C_J(Ax^t),\qquad
    e_t^y=\frac1{k_y}\sum_{j\in J}\rho_{t,j},\qquad
    G_y=I_{\ell_y}+\frac1{k_y}\one\one^\top.
\]
By eliminating \(\lambda_t\) exactly as in~\cref{lem:reflection-x}, we have 
\begin{align}
    G_y\Delta w_t
    &=
    -\eta d_{t+1}^y-\eta e_t^y\one+\eta(\rho_t)_{O_y},
    \label{eq:reflection-vector-y}\\
    \eta(\rho_t)_{O_y}^\top w_t
    &=
    -\lVert\Delta w_t\rVert_{G_y}^2
    -\eta(d_{t+1}^y)^\top\Delta w_t
    -\eta e_t^y\Delta M_t^y.
    \label{eq:reflection-scalar-y}
\end{align}
Considering the time-index changes, we define the telescoping term for the $y$-player: 
\begin{equation}
    B_{y,t}=\eta(d_t^y)^\top w_t.
    \label{eq:storage-y}
\end{equation}
Then, we can again write a decomposition of the term $-\eta(d_{t+1}^y)^\top\Delta w_t$: 
\begin{equation}
    -\eta(d_{t+1}^y)^\top\Delta w_t
    =
    B_{y,t}-B_{y,t+1}
    +\eta(d_{t+1}^y-d_t^y)^\top w_t,
    \label{eq:storage-telescope-y}
\end{equation}
and 
\[
    d_{t+1}^y-d_t^y=\mathsf C_J(A\Delta x_t).
\]
Hence, as in~\cref{eq:drift-x},
\begin{equation}
    \lVert d_{t+1}^y-d_t^y\rVert_\infty
    \leq
    \sqrt{1+\frac1{k_y}}\,\eta L^2.
    \label{eq:drift-y}
\end{equation}

For the on-support component, we symmetrically define 
\[
    h_t^y=(\Delta y_t)_J,\qquad
    \overline h_t^y=\frac1{k_y}\one^\top h_t^y.
\]
Then, we have 
\[
    \Delta M_t^y=-\one^\top h_t^y,\qquad
    (\rho_t)_J^\top y_J^t=-(\rho_t)_J^\top h_t^y,
\]
and therefore 
\begin{equation}
    (\rho_t)_J^\top y_J^t-e_t^y\Delta M_t^y
    =
    (\rho_t)_J^\top
    \left(\overline h_t^y\one-h_t^y\right).
    \label{eq:cancellation-y}
\end{equation}
Then, we can show the following two upper bounds for the $y$-player block: 
\[
    M_t^y\leq\frac{P_t^y}{\sigma_y},
    \qquad
    E_t^y\geq
    \underline y\sum_{j\in J}\rho_{t,j}.
\]
Combining these two bounds with~\cref{eq:reflection-scalar-y,eq:storage-telescope-y,eq:drift-y,eq:cancellation-y} yields 
\begin{equation}
    \boxed{
    \begin{aligned}
    \eta\rho_t^\top y^t
    \leq{}&
    B_{y,t}-B_{y,t+1}\\
    &+
    \sqrt{1+\frac1{k_y}}\,
    \frac{\eta L^2}{\sigma_y}(\eta P_t^y)
    +
    \sqrt{1-\frac1{k_y}}\,
    \frac{\eta L}{\underline y}(\eta E_t^y).
    \end{aligned}}
    \label{eq:block-bound-y}
\end{equation}
When \(J^c=\varnothing\), we define 
\[
    e_t^y=\frac1{k_y}\sum_{j\in J}\rho_{t,j},
    \qquad
    h_t^y=(\Delta y_t)_J,
    \qquad
    \overline h_t^y=\frac1{k_y}\one^\top h_t^y,
\]
and set \(B_{y,t}=0\). 
Again, it can be verified that~\cref{eq:block-bound-y} still holds by the same arguments as those showing the upper bound involving $E^y_t$.

\subsubsection{Combining the two blocks}

By choosing an appropriate stepsize satisfying~\cref{eq:stepsize}, we can make the coefficients before the dissipation terms be at most $\frac{1}{2}$.

Whenever the set of the off-support coordinates is nonempty, 
the definition of \(\delta\) in~\cref{eq:delta} gives
\(\eta L^2/\sigma_x\leq1/(2\sqrt2)\), and
\(\sqrt{1+1/k_x}\leq\sqrt2\). 
The same bounds hold on the \(y\)-player block. 
Hence, we have 
\begin{equation}
    \sqrt{1+\frac1{k_x}}\,
    \frac{\eta L^2}{\sigma_x}
    \leq\frac12,
    \qquad
    \sqrt{1+\frac1{k_y}}\,
    \frac{\eta L^2}{\sigma_y}
    \leq\frac12.
    \label{eq:absorb-P}
\end{equation}
Similarly, \(\eta L\leq\delta/(2\sqrt2)\) and $\delta \leq \min\{ \underline{x}, \underline{y} \}$ give
\begin{equation}
    \sqrt{1-\frac1{k_x}}\,
    \frac{\eta L}{\underline x}
    \leq\frac1{2\sqrt2}<\frac12,
    \qquad
    \sqrt{1-\frac1{k_y}}\,
    \frac{\eta L}{\underline y}
    \leq\frac1{2\sqrt2}<\frac12.
    \label{eq:absorb-E}
\end{equation}
Let
\[
    B_t:=B_{x,t}+B_{y,t}
\]
be the combined telescoping term. 
Combining the block bounds \cref{eq:block-bound-x,eq:block-bound-y} and 
applying~\cref{eq:absorb-P,eq:absorb-E} provides the important bound we need: 
\begin{align}
    r_t
    &\leq
    B_t-B_{t+1}
    +\frac{\eta}{2}P_t+\frac{\eta}{2}E_t
    \notag\\
    &\leq
    B_t-B_{t+1}+\frac12\mathscr D_t, 
    \label{eq:pointwise-bound}
\end{align}
where the second inequality holds by the definition of the dissipation in~\cref{eq:dissipation} and the nonnegativity of $P_{t+1}$ and $E_t$: 
\[
    \frac12\mathscr D_t
    -
    \left(
      \frac{\eta}{2}P_t+\frac{\eta}{2}E_t
    \right)
    =
    \frac{\eta}{2}P_{t+1}+\frac{\eta}{2}E_t
    \geq0.
\]

The term $B_t$ is uniformly bounded: 
Since each contrast payoff vector satisfies
\(\lVert d_t^x\rVert_\infty\leq\sqrt{1+1/k_x}\,L\) by
\eqref{eq:contrast-norm}, and each off-support probability is at most one, we have 
\[
    |B_{x,t}|
    \leq
    \sqrt{1+\frac1{k_x}}\,q
    \leq\sqrt2\,q,
\]
and similarly \(|B_{y,t}|\leq\sqrt2\,q\). 
Hence, 
\begin{equation}
    |B_t|\leq2\sqrt2\,q.
    \label{eq:storage-bound}
\end{equation}

\begin{remark} 
Note that the fixed index sets \(I\) and \(J\) are supports of the
Goldman-Tucker saddle point, not of the current iterates. 
The arguments in this section hold when the iterate coordinates repeatedly leave and reenter the simplex boundary.
\end{remark}

\subsection{Finite residual budget}
\label{sec:budget}

By combining the per-step bound~\cref{eq:pointwise-bound} with the exact energy balance~\cref{eq:energy-balance}, we can show that the total budget for the residual term is finite, i.e., the residual term is summable. 
Let
\[
    R_T=\sum_{t=0}^{T-1}r_t,
    \qquad
    \mathcal D_T=\sum_{t=0}^{T-1}\mathscr D_t
\]
denote the accumulated residual and dissipation. 
Summing up the per-step bound~\cref{eq:pointwise-bound} and then
applying the bound~\cref{eq:storage-bound} yields 
\begin{equation}
    R_T
    \leq
    B_0-B_T+\frac12\mathcal D_T
    \leq
    4\sqrt2\,q+\frac12\mathcal D_T.
    \label{eq:residual-absorption}
\end{equation}
The exact energy balance~\cref{eq:energy-balance}, summed over the
same range, provides 
\begin{equation}
    \mathcal D_T=R_T-V_T+V_0.
    \label{eq:dissipation-sum}
\end{equation}
By substituting~\cref{eq:dissipation-sum} into~\cref{eq:residual-absorption} and solving for \(R_T\), we have 
\[
    R_T
    \leq
    8\sqrt2\,q+V_0-V_T.
\]
Using the energy bounds~\cref{eq:energy-lower,eq:energy-initial}, we prove 
\begin{equation}
    \boxed{
    R_T\leq4+(2+8\sqrt2)q.}
    \label{eq:residual-budget}
\end{equation}
This inequality shows that the total contribution of $r_t$ or the ``projection collisions'' is bounded by an absolute constant, 
for every horizon $T$ and every initialization $(x^0, y^0)$.

\subsection{From finite residual budget to the rate on the averaged gap}
\label{sec:gap}

It remains to bound the duality gap of the averages of the iterates by the accumulated residual. 
This can be done by using the standard projection inequality: if
\(p^+=\Pi_C(p+\eta g)\) for a closed convex set \(C\), then for every
\(z\in C\),
\begin{equation}
    \eta g^\top(z-p^+)
    \leq
    \frac12\left(
      \lVert p-z\rVert_2^2
      -\lVert p^+-z\rVert_2^2
      -\lVert p^+-p\rVert_2^2
    \right).
    \label{eq:projection-vi}
\end{equation}

Fix arbitrary comparators \(x\in\Delta_n\) and \(y\in\Delta_m\), and
let
\[
    g_t:=y^\top Ax^t-(y^t)^\top Ax.
\]
Averaging \(g_t\) over \(t=1,\dots,T\) and taking the supremum over comparators produces 
\(\Gap(\overline x^T,\overline y^T)\).

Because the two players update in alternation, 
we use the telescoping potential functions that pair the current \(x\)-iterate 
\emph{once with the current} and \emph{once with the preceding
\(y\)-iterate}: 
for \(t\geq1\), define
\begin{align}
    \phi_t
    &=
    \frac12\lVert x^t-x\rVert_2^2
    +\frac12\lVert y^t-y\rVert_2^2
    +\eta(y^t)^\top Ax,
    \label{eq:phi}\\
    \psi_t
    &=
    \frac12\lVert x^t-x\rVert_2^2
    +\frac12\lVert y^{t-1}-y\rVert_2^2
    -\frac12\lVert y^t-y^{t-1}\rVert_2^2,
    \label{eq:psi}\\
    \Phi_t&=\phi_t+\psi_t.
    \label{eq:Phi}
\end{align}

Applying~\cref{eq:projection-vi} to the \(x\)-update with comparator
\(x \in \Delta_n\) and to the \(y\)-update with comparator \(y \in \Delta_m\), 
and assembling the bilinear terms, it follows that 
\begin{equation}
    \eta g_{t+1}
    \leq
    \phi_t-\phi_{t+1}
    +\eta u_t^\top\Delta y_t
    -\frac12\left(
      \lVert\Delta x_t\rVert_2^2
      +\lVert\Delta y_t\rVert_2^2
    \right).
    \label{eq:shifted-one}
\end{equation}
Instead, pairing the current \(x\)-update with the preceding \(y\)-update leads to 
\begin{equation}
    \eta g_t
    \leq
    \psi_t-\psi_{t+1}
    +\eta v_t^\top\Delta x_t
    -\frac12\left(
      \lVert\Delta x_t\rVert_2^2
      +\lVert\Delta y_t\rVert_2^2
    \right).
    \label{eq:shifted-two}
\end{equation}
Adding up~\cref{eq:shifted-one,eq:shifted-two} yields the residual term in~\cref{eq:residual-definition}:
\begin{equation}
    \eta(g_t+g_{t+1})
    \leq
    \Phi_t-\Phi_{t+1}+r_t,
    \qquad t\geq1.
    \label{eq:combined-shifted}
\end{equation}

For \(T\geq2\), summing \eqref{eq:combined-shifted} from \(t=1\) to
\(T-1\) counts every interior \(g_t\) twice and the endpoints once, so
\begin{equation}
    2\eta\sum_{t=1}^Tg_t
    \leq
    \Phi_1-\Phi_T
    +\sum_{t=1}^{T-1}r_t
    +\eta(g_1+g_T).
    \label{eq:gap-telescope}
\end{equation}
The following bounds show that the endpoint quantities are bounded uniformly over the comparators:
\begin{equation}
    \Phi_1\leq4+q,\qquad
    -\Phi_T\leq q+\frac{q^2}{2},\qquad
    \eta(g_1+g_T)\leq4q, 
    \label{eq:endpoints}
\end{equation}
where the first bound follows from the simplex squared-diameter bound $2$ together with \(|(y^1)^\top Ax|\leq L\),  the second inequality holds because the only possibly negative terms in \(\Phi_T\): \(-\lVert\Delta y_{T-1}\rVert_2^2/2\) and \(\eta(y^T)^\top Ax\) are bounded below by \(-q^2/2\) and \(-q\) respectively, 
and the third bound uses \(|g_t|\leq2L\).

Because every residual is nonnegative by Lemma~\ref{lem:residual}, we have 
\[
    \sum_{t=1}^{T-1}r_t\leq R_T,
\]
and then we can apply the finite budget bound~\cref{eq:residual-budget}. 
Combining it with the endpoint bounds~\cref{eq:endpoints} in~\cref{eq:gap-telescope}, and then using \(q\leq1/(2\sqrt2)\) from~\cref{eq:q-bound}, we obtain
\begin{align}
    2\eta\sum_{t=1}^Tg_t
    &\leq
    8+(8+8\sqrt2)q+\frac{q^2}{2}
    \notag\\
    &\leq
    12+2\sqrt2+\frac1{16}
    =
    \frac{193}{16}+2\sqrt2
    <15.
    \label{eq:final-constant}
\end{align}
By dividing both sides of~\cref{eq:final-constant} by \(2\eta T\), taking the supremum over the comparators, and using the bilinearity
of the pairing $g_t$, we have 
\[
    \Gap(\overline x^T,\overline y^T)
    \leq\frac{15}{2\eta T},
\]
which completes~\cref{thm:main} for \(T\geq2\).

For \(T=1\), the average is just the first iterate, and the direct bound suffices:
\[
    \Gap(x^1,y^1)
    \leq2L
    =
    \frac{2q}{\eta}
    \leq
    \frac1{\sqrt2\,\eta}
    <
    \frac{15}{2\eta}.
\]
Therefore, \cref{thm:main} is proven for every \(T\geq1\).

\section{Lyapunov-LMI Details}\label{sec:appendix-lmi}
In this appendix we describe the Lyapunov LMI that is at the core of the PEP
framework which motivated the convergence analysis of \cref{thm:main-text}. Also we 
present a proof of 
\cref{thm:c123-lmi-conditional-rate}, which is conditional on having a exactly feasible solution to the LMI in consideration. Throughout, the setting consists of compact convex
sets \(\mathcal X\subseteq\R^n\) and \(\mathcal Y\subseteq\R^m\) of
radius \(D\), a payoff matrix normalized by
\(\lVert A\rVert_2\leq1\), and AltGDA with projections onto \(\mathcal X\) and \(\mathcal Y\).

\paragraph{Local notation.}
Notation that we use in this appendix is as follows.
\(\mathbb S^d\) (respectively \(\mathbb S^d_+\)) denotes the symmetric
(respectively symmetric positive semidefinite) \(d\times d\) matrices,
and \(M\succeq0\) means \(M\in\mathbb S^d_+\).  The symmetric outer
product of \(u,v\in\R^d\) is
\(\boxdot(u,v)=\tfrac12\left(uv^\top+vu^\top\right)\).
For a closed convex set \(C\), \(\iota_C\) is its indicator function,
so \(\partial\iota_C(z)\) is the normal cone to \(C\) at \(z\).
Several symbols reuse letters used elsewhere in the paper, always
with a new, locally defined meaning.  The starred pair
\((x^\star,y^\star)\) is an arbitrary comparator in
\(\mathcal X\times\mathcal Y\), not the strictly complementary saddle
point used to prove \cref{thm:main-text}.  The symbols \(\mu\) and
\(\lambda\) denote LMI variables, \emph{not} the projection multipliers of
Appendix~\ref{sec:appendix-proof}.  The vectors
\(\bar\delta,\underline\delta\) are projection residuals and they are not related to the separation parameter \(\delta\).

\subsection{Interpolation results}

In this section and the next, we will be using the following interpolation results from \citealp{bousselmi2024interpolation} and \citealp[Theorem~3.6]{taylor2017exact}, respectively.

\begin{lemma}[{Interpolation of a matrix with bounded singular value,
\citealp{bousselmi2024interpolation}}]
\label{lem:matrix-interpolation}
Consider the sets of pairs
\(\{(x^i,p^i)\}_{i\in\{1,\dots,T_1\}}\subseteq\R^n\times\R^m\) and
\(\{(y^j,q^j)\}_{j\in\{1,\dots,T_2\}}\subseteq\R^m\times\R^n\), and
define the matrices
\(X=[x^1\mid\cdots\mid x^{T_1}]\in\R^{n\times T_1}\),
\(P=[p^1\mid\cdots\mid p^{T_1}]\in\R^{m\times T_1}\),
\(Y=[y^1\mid\cdots\mid y^{T_2}]\in\R^{m\times T_2}\), and
\(Q=[q^1\mid\cdots\mid q^{T_2}]\in\R^{n\times T_2}\).  Then there
exists a matrix \(A\in\R^{m\times n}\) with maximum singular value
\(\sigma_{\textup{max}}(A)\leq L\) such that \(p^i=Ax^i\) for all
\(i\in\{1,\dots,T_1\}\) and \(q^j=A^\top y^j\) for all
\(j\in\{1,\dots,T_2\}\) if and only if
\begin{align*}
 & X^\top Q=P^\top Y,\\
 & L^2X^\top X-P^\top P\succeq0,\\
 & L^2Y^\top Y-Q^\top Q\succeq0.
\end{align*}
\end{lemma}

\begin{lemma}[{Interpolation of a compact convex set with bounded
radius, \citealp[Theorem~3.6]{taylor2017exact}}]
\label{lem:pep-set-interp}
Let \(\mathcal I\) be a finite index set and
\(\{(z^i,g^i)\}_{i\in\mathcal I}\subseteq\R^d\times\R^d\).  There
exists a compact convex set \(C\subseteq\R^d\) contained in the ball
of radius \(R\) about the origin with \(z^i\in C\) and
\(g^i\in\partial\iota_C(z^i)\) for all \(i\in\mathcal I\) if and
only if
\begin{align*}
 & (g^j)^\top(z^i-z^j)\leq0,\quad i,j\in\mathcal I,\\
 & \lVert z^i\rVert_2^2\leq R^2,\quad i\in\mathcal I.
\end{align*}
\end{lemma}

The bounded-radius assumption is not part of the LMI itself; it is
used only at the end of the proof to bound the initial endpoint
potential.

\subsection{Derivation of the Lyapunov LMI}
\label{subsec:c123-lmi-specification}

\paragraph{Window notation.}
Fix a base index \(k\), an integer \(h\geq0\), and a normalized matrix
satisfying \(\lVert A\rVert_2\leq1\).  Let \(p^i=Ax^i\) and
\(q^i=A^\top y^i\).  For the projection residuals, write
\(\bar\delta^{i+1}=x^i-\eta q^i-x^{i+1}
\in\partial\iota_{\mathcal X}(x^{i+1})\) and
\(\underline\delta^{i+1}=y^i+\eta p^{i+1}-y^{i+1}
\in\partial\iota_{\mathcal Y}(y^{i+1})\), with
\(\bar\delta^0=0\) and \(\underline\delta^0=0\).  These two zero
residuals are padding terms: they make the \(k=0\) window match the
shifted-window notation used below.  For a comparator
\((x^\star,y^\star)\in\mathcal X\times\mathcal Y\), set
\(p^\star=Ax^\star\), \(q^\star=A^\top y^\star\), and
\(\bar\delta^\star=\underline\delta^\star=0\).  Write
\(I^\star_r(k)=\{\star,k,k+1,\dots,k+r\}\).

\paragraph{Potential states.}
The numerical certificate is designed to search over Lyapunov matrices
\(Q^x\in\mathbb S^{2h+6}\) and \(Q^y\in\mathbb S^{2h+7}\); these are
required to be symmetric but need not be positive semidefinite a
priori.  For a fixed comparator, define
\[
H^x_k=\left[x^\star\mid q^\star\mid x^k\mid\bar\delta^k\mid\cdots\mid
\bar\delta^{k+h+1}\mid q^k\mid\cdots\mid q^{k+h}\right],
\]
\[
H^y_k=\left[y^\star\mid p^\star\mid y^k\mid\underline\delta^k\mid
\cdots\mid\underline\delta^{k+h+1}\mid p^k\mid\cdots\mid
p^{k+h+1}\right],
\]
and define \(H^x_{k+1}\) and \(H^y_{k+1}\) by shifting the iterate,
residual, and gradient indices forward by one.  Thus
\(H^x_k\) and \(H^x_{k+1}\) have \(2h+6\) columns and
\(H^y_k,H^y_{k+1}\) have
\(2h+7\); the shifted pair repeats the column pattern of the
unshifted pair with base index \(k+1\).  The potential is
\[
\mathcal V_k
=\mathbf{tr}\left(Q^x(H^x_k)^\top H^x_k\right)
+\mathbf{tr}\left(Q^y(H^y_k)^\top H^y_k\right),
\]
written \(\mathcal V_k(x^\star,y^\star)\) when the dependence on the
comparator columns of \(H^x_k,H^y_k\) matters.

\paragraph{Full-state selectors.}
Let \(N_x=2h+8\) and \(N_y=2h+9\).  Define the full-state data
matrices
\[
F^x_k=\left[x^\star\mid q^\star\mid x^k\mid\bar\delta^k\mid\cdots\mid
\bar\delta^{k+h+2}\mid q^k\mid\cdots\mid q^{k+h+1}\right]
\in\R^{n\times N_x},
\]
and
\[
F^y_k=\left[y^\star\mid p^\star\mid y^k\mid\underline\delta^k\mid
\cdots\mid\underline\delta^{k+h+2}\mid p^k\mid\cdots\mid
p^{k+h+2}\right]\in\R^{m\times N_y}.
\]
Their Gram matrices are \(G^x=(F^x_k)^\top F^x_k\) and
\(G^y=(F^y_k)^\top F^y_k\).  Hatted symbols denote the corresponding
coordinate selectors in these full states: a hatted selector
\(\widehat z\in\R^{N_x}\) represents the actual vector
\(F^x_k\widehat z\) on the \(x\)-side, and analogously on the
\(y\)-side.  For the \(x\)-side, set
\(\widehat x^\star=e^{N_x}_1\), \(\widehat q^\star=e^{N_x}_2\),
\(\widehat x^k=e^{N_x}_3\), \(\widehat{\bar\delta}^\star=0^{N_x}\),
\(\widehat{\bar\delta}^{k+i}=e^{N_x}_{4+i}\) for
\(i\in\{0,\dots,h+2\}\), \(\widehat q^{k+i}=e^{N_x}_{h+7+i}\) for
\(i\in\{0,\dots,h+1\}\), and
\[
\widehat x^{k+i}
=\widehat x^k-\sum_{t=1}^i\widehat{\bar\delta}^{k+t}
-\eta\sum_{t=0}^{i-1}\widehat q^{k+t},
\qquad i\in\{1,\dots,h+2\}.
\]
For the \(y\)-side, set \(\widehat y^\star=e^{N_y}_1\),
\(\widehat p^\star=e^{N_y}_2\), \(\widehat y^k=e^{N_y}_3\),
\(\widehat{\underline\delta}^\star=0^{N_y}\),
\(\widehat{\underline\delta}^{k+i}=e^{N_y}_{4+i}\) for
\(i\in\{0,\dots,h+2\}\), \(\widehat p^{k+i}=e^{N_y}_{h+7+i}\) for
\(i\in\{0,\dots,h+2\}\), and
\[
\widehat y^{k+i}
=\widehat y^k-\sum_{t=1}^i\widehat{\underline\delta}^{k+t}
+\eta\sum_{t=1}^i\widehat p^{k+t},
\qquad i\in\{1,\dots,h+2\}.
\]
Note that, for each potential-state matrix, the hatted version denotes the
selector matrix that is obtained columnwise from the corresponding hatted
vectors, e.g., \(H^x_k=F^x_k\widehat H^x_k\) and
\(H^y_k=F^y_k\widehat H^y_k\), with the same convention for the
shifted matrices. Let us now define
\[
\widehat X=\left[\widehat x^\star\mid\widehat x^k\mid\widehat x^{k+1}
\mid\cdots\mid\widehat x^{k+h+2}\right],\qquad
\widehat P=\left[\widehat p^\star\mid\widehat p^k\mid\widehat p^{k+1}
\mid\cdots\mid\widehat p^{k+h+2}\right],
\]
\[
\widehat Y=\left[\widehat y^\star\mid\widehat y^k\mid\widehat y^{k+1}
\mid\cdots\mid\widehat y^{k+h+1}\right],\qquad
\widehat Q=\left[\widehat q^\star\mid\widehat q^k\mid\widehat q^{k+1}
\mid\cdots\mid\widehat q^{k+h+1}\right].
\]
The symbol \(\widehat Q\) here is a selector matrix for the
\(q\)-vectors and is unrelated to the Lyapunov matrices \(Q^x\) and
\(Q^y\).

\paragraph{Generic condition.}
We now build the algebraic certificate from a generic trace condition.
For symmetric target matrices \(M_x\in\mathbb S^{N_x}\) and
\(M_y\in\mathbb S^{N_y}\), we can write the desired trace condition
\(\mathbf{tr}(M_xG^x)+\mathbf{tr}(M_yG^y)\leq0\) for Gram matrices $G^x, G^y$ associated with an admissible AltGDA window.  The 
worst-case problem computes the largest value of this target over all
AltGDA trajectories in the admissible window and all normalized matrices compatible within this window:
\[
\mathsf p_{\mathrm{op}}(M_x,M_y):=\left(\begin{array}{l}
\underset{\substack{F^x_k\in\R^{n\times N_x},\ F^y_k\in\R^{m\times N_y},\\
A\in\R^{m\times n},\ m,n\in\mathbb N}}{\mbox{maximize}}\;
\mathbf{tr}\!\left(M_x(F^x_k)^\top F^x_k\right)
+\mathbf{tr}\!\left(M_y(F^y_k)^\top F^y_k\right)\\
\textup{subject to}\\
\lVert A\rVert_2\leq1,\\
F^y_k\widehat p^i=AF^x_k\widehat x^i,\quad i\in I^\star_{h+2}(k),\\
F^x_k\widehat q^j=A^\top F^y_k\widehat y^j,\quad j\in I^\star_{h+1}(k),\\
\ip{F^x_k\widehat{\bar\delta}^j}{F^x_k(\widehat x^i-\widehat x^j)}
\leq0,\quad i,j\in I^\star_{h+2}(k),\quad j\notin\{i,\star\},\\
\ip{F^y_k\widehat{\underline\delta}^j}{F^y_k(\widehat y^i-\widehat y^j)}
\leq0,\quad i,j\in I^\star_{h+2}(k),\quad j\notin\{i,\star\}.
\end{array}\right)
\]

In the problem above, the last two constraints model convexity of the sets in consideration. The hatted selectors already encode the AltGDA recursions, so we do not need any 
separate update equations for them.  The first two
constraint lines impose \(p^i=Ax^i\) and \(q^j=A^\top y^j\) and the last
two are the normal-cone inequalities for the two projections (with zero cases omitted).

By \cref{lem:matrix-interpolation} (with \(L=1\)), the explicit matrix
\(A\) can be eliminated: the problem is equivalently written with the
interpolation constraints
\((F^x_k\widehat X)^\top(F^x_k\widehat Q)
=(F^y_k\widehat P)^\top(F^y_k\widehat Y)\),
\((F^y_k\widehat P)^\top(F^y_k\widehat P)
\preceq(F^x_k\widehat X)^\top(F^x_k\widehat X)\), and
\((F^x_k\widehat Q)^\top(F^x_k\widehat Q)
\preceq(F^y_k\widehat Y)^\top(F^y_k\widehat Y)\) in place of the first
three constraint lines.  Passing to
\(G^x=(F^x_k)^\top F^x_k\) and \(G^y=(F^y_k)^\top F^y_k\) and dropping
the implicit rank restrictions gives the homogeneous Gram relaxation
\[
\mathsf{p}(M_{x},M_{y}):=\left(\begin{array}{l}
\underset{\substack{G^{x}\in\mathbb{S}^{N_{x}}_{+}\\
G^{y}\in\mathbb{S}^{N_{y}}_{+}
}
}{}{\mbox{maximize}}\;\mathbf{tr}(M_{x}G^{x})+\mathbf{tr}(M_{y}G^{y})\\
\textup{subject to}\\
\mathbf{tr}\!\left(\boxdot(\widehat{x}^{i},\widehat{q}^{j})G^{x}\right)-\mathbf{tr}\!\left(\boxdot(\widehat{p}^{i},\widehat{y}^{j})G^{y}\right)=0,\quad i\in I^{\star}_{h+2}(k),\quad j\in I^{\star}_{h+1}(k),\\
\widehat{P}^{\top}G^{y}\widehat{P}\preceq\widehat{X}^{\top}G^{x}\widehat{X},\\
\widehat{Q}^{\top}G^{x}\widehat{Q}\preceq\widehat{Y}^{\top}G^{y}\widehat{Y},\\
\mathbf{tr}\!\left(\boxdot(\widehat{\bar{\delta}}^{j},\widehat{x}^{i}-\widehat{x}^{j})G^{x}\right)\leq0,\quad i,j\in I^{\star}_{h+2}(k),\quad j\notin\{i,\star\},\\
\mathbf{tr}\!\left(\boxdot(\widehat{\underline{\delta}}^{j},\widehat{y}^{i}-\widehat{y}^{j})G^{y}\right)\leq0,\quad i,j\in I^{\star}_{h+2}(k),\quad j\notin\{i,\star\},
\end{array}\right)
\]
where the first three constraints use the Lemma \ref{lem:matrix-interpolation} last two constraints use Lemma \ref{lem:pep-set-interp}. The feasible set of this Gram relaxation contains every Gram pair
generated by the preceding vector formulations, so nonpositivity of
\(\mathsf p(M_x,M_y)\) is a sufficient certificate for the generic
trace condition.

\paragraph{Compact dual certificate for a generic condition.}
The maximization problem above is homogeneous, so a zero-objective
dual feasible point certifies \(\mathsf p(M_x,M_y)\leq0\).
Concretely, take PSD blocks \(U\in\mathbb S^{h+4}_+\),
\(V\in\mathbb S^{h+3}_+\), \(Z^x\in\mathbb S^{N_x}_+\), and
\(Z^y\in\mathbb S^{N_y}_+\); free multipliers \(\mu_{i,j}\) for
\(i\in I^\star_{h+2}(k)\) and \(j\in I^\star_{h+1}(k)\); and
nonnegative multipliers \(\lambda^x_{i,j},\lambda^y_{i,j}\geq0\) for
\(i,j\in I^\star_{h+2}(k)\) with \(j\notin\{i,\star\}\).  It is
enough to find such data satisfying
\begin{align*}
 & R^\lambda_x=\sum_{\substack{i,j\in I^\star_{h+2}(k)\\
   j\notin\{i,\star\}}}
   \lambda^x_{i,j}\boxdot\!\left(\widehat{\bar\delta}^j,
   \widehat x^i-\widehat x^j\right),
 \qquad
   R^\lambda_y=\sum_{\substack{i,j\in I^\star_{h+2}(k)\\
   j\notin\{i,\star\}}}
   \lambda^y_{i,j}\boxdot\!\left(\widehat{\underline\delta}^j,
   \widehat y^i-\widehat y^j\right),\\
 & R^\mu_x=\sum_{\substack{i\in I^\star_{h+2}(k)\\
   j\in I^\star_{h+1}(k)}}
   \mu_{i,j}\,\boxdot(\widehat x^i,\widehat q^j),
 \qquad
   R^\mu_y=\sum_{\substack{i\in I^\star_{h+2}(k)\\
   j\in I^\star_{h+1}(k)}}
   \mu_{i,j}\,\boxdot(\widehat p^i,\widehat y^j),\\
 & -M_x+R^\lambda_x+R^\mu_x-\widehat XU\widehat X^\top
   +\widehat QV\widehat Q^\top=Z^x,\\
 & -M_y+R^\lambda_y-R^\mu_y+\widehat PU\widehat P^\top
   -\widehat YV\widehat Y^\top=Z^y.
\end{align*}
We now take trace inner products of the last two identities with any feasible
\(G^x,G^y\).  Note that the operator-equality constraints cancel the \(\mu\)
terms, the normal-cone inequalities make the \(\lambda\) terms
nonpositive, the two matrix inequalities do the same for the \(U,V\)
terms, and, finally, the slack terms satisfy
\(\mathbf{tr}(Z^xG^x)+\mathbf{tr}(Z^yG^y)\geq0\). As a result, we have
\(\mathbf{tr}(M_xG^x)+\mathbf{tr}(M_yG^y)\leq0\) for every feasible
$G^x, G^y$.

\paragraph{Instantiating the three conditions.}
The full \(C_1\)-\(C_2\)-\(C_3\) certificate searches over symmetric
auxiliary matrices \(S^x\in\mathbb S^{N_x}\) and
\(S^y\in\mathbb S^{N_y}\).  Let us now define
\(\mathcal R^S_k=\mathbf{tr}(S^xG^x)+\mathbf{tr}(S^yG^y)\) and
\begin{align*}
 & \mathcal Q^x_0=\widehat H^x_kQ^x(\widehat H^x_k)^\top,\qquad
   \mathcal Q^x_1=\widehat H^x_{k+1}Q^x(\widehat H^x_{k+1})^\top,\\
 & \mathcal Q^y_0=\widehat H^y_kQ^y(\widehat H^y_k)^\top,\qquad
   \mathcal Q^y_1=\widehat H^y_{k+1}Q^y(\widehat H^y_{k+1})^\top.
\end{align*}

The three target pairs can be written as:
\begin{align*}
 & M^{C_1}_x=\mathcal Q^x_1-\mathcal Q^x_0+S^x,\qquad
   M^{C_1}_y=\mathcal Q^y_1-\mathcal Q^y_0+S^y,\\
 & M^{C_2}_x=-\mathcal Q^x_0,\qquad
   M^{C_2}_y=-\mathcal Q^y_0,\\
 & M^{C_3}_x=\boxdot(\widehat q^\star,\widehat x^k)-S^x,\qquad
   M^{C_3}_y=-\boxdot(\widehat p^\star,\widehat y^k)-S^y.
\end{align*}
For each condition \(r\in\{C_1,C_2,C_3\}\), take an independent copy
of the generic dual variables, writing \(R^\lambda_{x,r}\),
\(R^\lambda_{y,r}\), \(R^\mu_{x,r}\), \(R^\mu_{y,r}\) for the generic
sums with \(\lambda^x,\lambda^y,\mu\) replaced by
\(\lambda^{r,x},\lambda^{r,y},\mu^r\).  The \emph{Lyapunov LMI} is
the collection of the six affine identities
\begin{align}
-M^r_x+R^\lambda_{x,r}+R^\mu_{x,r}-\widehat XU^r\widehat X^\top
+\widehat QV^r\widehat Q^\top &= Z^{r,x},
\label{eq:c123-lmi-x}\\
-M^r_y+R^\lambda_{y,r}-R^\mu_{y,r}+\widehat PU^r\widehat P^\top
-\widehat YV^r\widehat Y^\top &= Z^{r,y},
\label{eq:c123-lmi-y}
\end{align}
for \(r\in\{C_1,C_2,C_3\}\), in the variables
\begin{align*}
 & Q^x\in\mathbb S^{2h+6},\ Q^y\in\mathbb S^{2h+7},\
   S^x\in\mathbb S^{N_x},\ S^y\in\mathbb S^{N_y},\\
 & U^r\in\mathbb S^{h+4}_+,\ V^r\in\mathbb S^{h+3}_+,\
   Z^{r,x}\in\mathbb S^{N_x}_+,\ Z^{r,y}\in\mathbb S^{N_y}_+,
   \quad r\in\{C_1,C_2,C_3\},\\
 & \mu^r_{i,j}\in\R,\quad i\in I^\star_{h+2}(k),\
   j\in I^\star_{h+1}(k),\ r\in\{C_1,C_2,C_3\},\\
 & \lambda^{r,x}_{i,j}\geq0,\ \lambda^{r,y}_{i,j}\geq0,\quad
   i,j\in I^\star_{h+2}(k),\ j\notin\{i,\star\},\
   r\in\{C_1,C_2,C_3\}.
\end{align*}
This is the combined feasibility problem solved in the numerical
diagnostics of \cref{sec:lmi}.

\paragraph{Use in the proof.}
The Gram matrices of an actual AltGDA window are feasible for the
relaxation above as the normal-cone and operator-consistency
constraints hold by construction and the interpolation constraints
hold by the necessity direction of \cref{lem:matrix-interpolation}. Thus we can apply the
preceding trace argument here to each
condition \(r\) and its identities
\eqref{eq:c123-lmi-x}--\eqref{eq:c123-lmi-y}, leading to
\(\mathbf{tr}(M^r_xG^x)+\mathbf{tr}(M^r_yG^y)\leq0\) on every actual
window.  Note that, while we state the identities here for fixed index \(k\), 
the selector pattern depends only on offsets from \(k\), so a single
feasible certificate yields these inequalities for every base window.
Finally, the target matrices for the three conditions encode the inequalities
\begin{align*}
 & \mathcal V_{k+1}-\mathcal V_k+\mathcal R^S_k\leq0,\\
 & \mathcal V_k\geq0,\\
 & (q^\star)^\top x^k-(p^\star)^\top y^k\leq\mathcal R^S_k,
\end{align*}
which the proof below combines by telescoping.

\subsection{Proof of the conditional theorem}

\begin{proof}[Proof of \cref{thm:c123-lmi-conditional-rate}]
Let \((x^\star,y^\star)\in\mathcal X\times\mathcal Y\) be an arbitrary
comparator, and set \(p^\star=Ax^\star\) and \(q^\star=A^\top y^\star\).
Let us now fix a base window \(k\). Due to the trace argument of the compact
dual-certificate construction applied to the condition-\(r\)
identities \eqref{eq:c123-lmi-x}--\eqref{eq:c123-lmi-y} on the actual
Gram matrices, we have
\[
\mathbf{tr}(M^r_xG^x)+\mathbf{tr}(M^r_yG^y)\leq0,
\qquad r\in\{C_1,C_2,C_3\}.
\]

For \(r=C_1\), the definition of \(M^{C_1}\) gives
\(\mathcal V_{k+1}(x^\star,y^\star)-\mathcal V_k(x^\star,y^\star)
+\mathcal R^S_k\leq0\).
For \(r=C_3\), the definition of \(M^{C_3}\) gives
\((q^\star)^\top x^k-(p^\star)^\top y^k-\mathcal R^S_k\leq0\).
Combining the two yields the one-step bound
\[
(q^\star)^\top x^k-(p^\star)^\top y^k
\leq
\mathcal V_k(x^\star,y^\star)-\mathcal V_{k+1}(x^\star,y^\star).
\]
Summing from \(k=0\) to \(T-1\) telescopes the potential.  For
\(\bar x_T=\frac1T\sum_{k=0}^{T-1}x^k\) and
\(\bar y_T=\frac1T\sum_{k=0}^{T-1}y^k\), the comparator terms satisfy
\[
\frac1T\sum_{k=0}^{T-1}
\left((q^\star)^\top x^k-(p^\star)^\top y^k\right)
=(y^\star)^\top A\bar x_T-\bar y_T^\top Ax^\star,
\]
and therefore
\[
T\left[(y^\star)^\top A\bar x_T-\bar y_T^\top Ax^\star\right]
\leq
\mathcal V_0(x^\star,y^\star)-\mathcal V_T(x^\star,y^\star).
\]
The \(C_2\) condition applied to the terminal window gives
\(-\mathcal V_T(x^\star,y^\star)\leq0\), so
\[
T\left[(y^\star)^\top A\bar x_T-\bar y_T^\top Ax^\star\right]
\leq
\mathcal V_0(x^\star,y^\star).
\]

For any symmetric matrix \(Q\) and any matrix \(H\),
\(\mathbf{tr}(QH^\top H)\leq\max\{\lambda_{\max}(Q),0\}
\lVert H\rVert_F^2\).  Thus
\[
\mathcal V_0(x^\star,y^\star)
\leq
\varsigma_x^+\lVert H^x_0\rVert_F^2
+\varsigma_y^+\lVert H^y_0\rVert_F^2.
\]
The radius assumptions and \(\lVert A\rVert_2\leq1\) give
\(\lVert p^i\rVert\leq D\) and \(\lVert q^i\rVert\leq D\).  Moreover,
each nonzero projection residual in the initial window has norm at
most \(\eta D\), because
\[
\lVert\bar\delta^{i+1}\rVert
\leq
\mathbf{dist}(x^i-\eta q^i,\mathcal X)
\leq
\eta\lVert q^i\rVert
\leq
\eta D,
\]
and similarly \(\lVert\underline\delta^{i+1}\rVert\leq\eta D\).  The
residual-padding convention sets
\(\bar\delta^0=\underline\delta^0=0\), so only \(h+1\) residual
columns contribute on each side.  Hence
\[
\lVert H^x_0\rVert_F^2\leq(h+4)D^2+(h+1)\eta^2D^2,
\qquad
\lVert H^y_0\rVert_F^2\leq(h+5)D^2+(h+1)\eta^2D^2.
\]
Substituting these bounds gives
\[
\mathcal V_0(x^\star,y^\star)
\leq
D^2\left((h+4)\varsigma_x^++(h+5)\varsigma_y^+
+(h+1)\eta^2(\varsigma_x^++\varsigma_y^+)\right),
\]
and the supremum over comparators completes the proof.
\end{proof}

\subsection{Computational setup}
\label{subsec:c123-lmi-diagnostics}

For each tested memory length \(h\in\{2,3,4,5\}\) and stepsize
\(\eta\), we modeled the final \(C_1, C_2, C_3\) feasibility problem
of \cref{subsec:c123-lmi-specification} using the JuMP package \citep{Lubin2023} in Julia and then solved it using MOSEK at its default tolerances. After each LMI numerical solve in Julia, we also independently computed three diagnostic terms representing quality of the solution: the maximum affine-identity residual over
\eqref{eq:c123-lmi-x}--\eqref{eq:c123-lmi-y}, the minimum eigenvalue
over all condition-specific PSD blocks
\(U^r,V^r,Z^{r,x},Z^{r,y}\), and the minimum of the multipliers
constrained to be nonnegative. A numerical solution is accepted if acceptance conditions described in \cref{sec:lmi} are satisfied.

\section{Computing the Last-Iterate Duality Gap for AltGDA via PEP}\label{sec:appendix-pep}

The finite-horizon PEP formulation for AltGDA presented in this section follows a framework similar to that of
\citet{nan2026convergence} but with the last-iterate gap as its objective. For completeness, we provide the full formulation here, prove the exactness of the upper bound using the counterexample described in \cref{prop:pep-exact-plateau} over compact convex sets, and then present simplex counterexample described in
\cref{ex:simplex-last-iterate}.

\paragraph{Notation.}
The spaces \(\mathbb S^d\) and
\(\mathbb S^d_+\) contain symmetric and symmetric positive
semidefinite \(d\times d\) matrices, respectively; \(M\succeq0\)
means \(M\in\mathbb S^d_+\).  For \(u,v\in\R^d\), set
\(\boxdot(u,v)=\tfrac12(uv^\top+vu^\top)\), so that
\(\mathbf{tr}(G\boxdot(u,v))=u^\top Gv\) for symmetric \(G\).  For a
closed convex set \(C\), \(\iota_C\) is its indicator and
\(\partial\iota_C(z)\) its normal cone at \(z\in C\).

The feasible sets \(\mathcal X\subseteq\R^n\) and
\(\mathcal Y\subseteq\R^m\) are compact convex sets contained in the
Euclidean balls of radii \(R_x\) and \(R_y\) about the origin, and
\(\lVert A\rVert_2\leq L\).  In our numerical experiments, we set \(L=R_x=R_y=1\).

As in Appendix~\ref{sec:appendix-lmi}, \(p^i=Ax^i\) and
\(q^i=A^\top y^i\) are operator images.  The pair \((x^\star,y^\star)\) is an arbitrary comparator
in \(\mathcal X\times\mathcal Y\), not the Goldman--Tucker saddle
point of \cref{thm:main-text}.  The full-horizon Gram blocks
\(G^x,G^y\) are distinct from both the window blocks of
Appendix~\ref{sec:appendix-lmi} and the reflected metrics \(G_x,G_y\)
of Appendix~\ref{sec:appendix-proof}.

The sampled-data index sets are
\[
\mathcal J_T=\{-3,-2,-1,0,1,\dots,T\},
\qquad
\mathcal I_T=\mathcal J_T\setminus\{-1\},
\]
and they differ from the supports \(I,J\) of \cref{sec:setting}.
Indices \(1,\dots,T\) label the iterates, \(0\) the initialization,
\(-2\) the comparator \((x^\star,y^\star)\), and \(-3\) an
index-weighted average.  Index \(-1\) is reserved for the all-ones
direction;
\(e_j\in\R^{2T+11}\) denotes a standard basis vector.

The code labels the player-side columns \((x,v)\) and \((u,y)\).  In
the paper's notation, the code variable \(u\) is the maximizer iterate
\(y\), the code variable \(y=Ax\) is the image \(p\), and
\(v=A^\top u\) is \(q\).  

\subsection{The finite-horizon worst-case problem}
\label{subsec:pep-inner}

The modeled class allows arbitrary \(m,n\), compact convex
\(\mathcal X,\mathcal Y\) with the stated radii, and
\(\lVert A\rVert_2\leq L\).  The AltGDA iterates satisfy \eqref{eq:altgda-main}.
With \(L=R_x=R_y=1\), this class contains every matrix-game
specialization of \eqref{eq:game-main} with
\(\lVert A\rVert_2\leq1\), because each simplex is a compact convex
subset of the unit ball.
We provide the exact compact-convex worst case and the independent simplex example
 in \cref{subsec:pep-exact-witnesses}. The horizon-\(T\) last-iterate gap is
\begin{equation}
\mathcal P_T(\eta)
=\left(\begin{array}{l}
\underset{\substack{\{(x^t,y^t)\}_{0\leq t\leq T},\ (x^\star,y^\star),\\
\mathcal X\subseteq\R^n,\ \mathcal Y\subseteq\R^m,\\
A\in\R^{m\times n},\ m,n\in\mathbb N}}{\mbox{maximize}}\;
(y^\star)^\top Ax^{T}-(y^{T})^\top Ax^\star\\[0.4em]
\textup{subject to}\\
\mathcal X\ \text{compact convex with radius }R_x,\quad
\mathcal Y\ \text{compact convex with radius }R_y,\\
\lVert A\rVert_2\leq L,\\
\{(x^t,y^t)\}_{1\leq t\leq T}\ \text{generated by AltGDA with
stepsize }\eta\text{ from }(x^0,y^0),\\
(x^0,y^0)\in\mathcal X\times\mathcal Y,\quad
(x^\star,y^\star)\in\mathcal X\times\mathcal Y.
\end{array}\right)
\label{eq:pep-inner}
\end{equation}
Note that the objective is the duality gap of the last-iterates \((x^T,y^T)\), as in
\eqref{eq:gap-main}.  The problem is infinite-dimensional at this stage due to the constraints associated with the sets, matrix, and dimensions.

\subsection{Span form of AltGDA}
\label{subsec:pep-span}

The trajectory constraint of \eqref{eq:pep-inner} has an equivalent
description in which the projections become normal-cone elements.

\begin{lemma}[Span form of AltGDA]
\label{lem:pep-span}
Fix \(\eta>0\), closed convex sets \(\mathcal X\subseteq\R^n\) and
\(\mathcal Y\subseteq\R^m\), a matrix \(A\in\R^{m\times n}\), and
\((x^0,y^0)\in\mathcal X\times\mathcal Y\).  A sequence
\(\{(x^t,y^t)\}_{1\leq t\leq T}\) satisfies
\(x^{t}=\Pi_{\mathcal X}(x^{t-1}-\eta A^\top y^{t-1})\) and
\(y^{t}=\Pi_{\mathcal Y}(y^{t-1}+\eta Ax^{t})\) for
\(t\in\{1,\dots,T\}\) if and only if there exist
\(f^t\in\partial\iota_{\mathcal X}(x^t)\) and
\(h^t\in\partial\iota_{\mathcal Y}(y^t)\) for \(t\in\{1,\dots,T\}\)
such that, with \(p^t=Ax^t\) and \(q^t=A^\top y^t\),
\begin{equation}
x^t=x^0-\eta\sum_{j=1}^{t}f^j-\eta\sum_{j=0}^{t-1}q^j,
\qquad
y^t=y^0-\eta\sum_{j=1}^{t}h^j+\eta\sum_{j=1}^{t}p^j,
\qquad t\in\{1,\dots,T\}.
\label{eq:pep-span}
\end{equation}
\end{lemma}

\begin{proof}
For a closed convex set \(C\), a point \(z\in C\), and any
\(w\in\R^d\), \(z=\Pi_C(w)\) if and only if
\(w-z\in\partial\iota_C(z)\)
\citep[Proposition~6.47]{bauschke2017convex}.  The \(x\)-update is
therefore equivalent to
\(x^{t-1}-\eta q^{t-1}-x^t\in\partial\iota_{\mathcal X}(x^t)\).
Because the normal cone is a cone and \(\eta>0\), this membership
holds exactly when the residual equals \(\eta f^t\) for some
\(f^t\in\partial\iota_{\mathcal X}(x^t)\), or equivalently,
\(x^t=x^{t-1}-\eta q^{t-1}-\eta f^t\).  Likewise, the \(y\)-update,
which uses \(p^t=Ax^t\), is equivalent to
\(y^t=y^{t-1}+\eta p^t-\eta h^t\) for some
\(h^t\in\partial\iota_{\mathcal Y}(y^t)\).  By summing these one-step
identities, we have \eqref{eq:pep-span}.  On the other hand, by differencing
\eqref{eq:pep-span} we recover the one-step identities, and using the same
characterization we can convert them back into the two projection updates.
\end{proof}

\subsection{Interpolation argument}
\label{subsec:pep-interp}

The formulation accesses \(\mathcal X,\mathcal Y\), and \(A\) only
through finitely many projections and matrix--vector products.
Using \cref{lem:pep-span}, \cref{lem:pep-set-interp}, and \cref{lem:matrix-interpolation}, we can characterize
exactly which finite collections of first-order data are realizable.

At the indices \(\mathcal I_T\), the formulation samples the iterates,
initialization, comparator
\((x^{-2},y^{-2})=(x^\star,y^\star)\), and index-weighted averages
\begin{equation}
x^{-3}=\frac{\sum_{j=1}^{T}j\,x^j}{\sum_{j=1}^{T}j},
\qquad
y^{-3}=\frac{\sum_{j=1}^{T}j\,y^j}{\sum_{j=1}^{T}j},
\label{eq:pep-wavg}
\end{equation}
which lie in \(\mathcal X\times\mathcal Y\) by convexity.  Each sample
has normal-cone data \(f^i,h^i\) and images \(p^i=Ax^i\),
\(q^i=A^\top y^i\).  Choosing zero normal-cone elements makes the
samples at \(i\in\{-3,-2,0\}\) valid.  The reserved index \(-1\)
adds one formal sample per side only for the operator families.

\subsection{Finite-dimensional reformulation}
\label{subsec:pep-finite}

Collect the sampled points and images over the full index set, with
columns ordered \(-3,-2,-1,0,1,\dots,T\):
\begin{align*}
X&=[x^i]_{i\in\mathcal J_T}\in\R^{n\times(T+4)},
&P&=[p^i]_{i\in\mathcal J_T}\in\R^{m\times(T+4)},\\
Y&=[y^i]_{i\in\mathcal J_T}\in\R^{m\times(T+4)},
&Q&=[q^i]_{i\in\mathcal J_T}\in\R^{n\times(T+4)}.
\end{align*}
In the sampled data the objective reads
\((q^{-2})^\top x^T-(y^T)^\top p^{-2}\).
Replacing the trajectory and class constraints of
\eqref{eq:pep-inner} by the sampled conditions produces a
finite-dimensional problem with the same optimal value:
\begin{equation}
\mathcal P_{T}(\eta)=\left(\begin{array}{l}
\underset{\substack{\{(x^i,f^i,q^i)\}_{i\in\mathcal J_T}\subseteq\R^n,\\
\{(y^i,h^i,p^i)\}_{i\in\mathcal J_T}\subseteq\R^m,\\
m,n\in\mathbb N}}{\mbox{maximize}}\;
(q^{-2})^\top x^{T}-(y^{T})^\top p^{-2}\\[0.4em]
\textup{subject to}\\
(f^j)^\top(x^i-x^j)\leq0,\quad
(h^j)^\top(y^i-y^j)\leq0,\quad i,j\in\mathcal I_T,\ i\neq j,\\
\lVert x^i\rVert_2^2\leq R_x^2,\quad
\lVert y^i\rVert_2^2\leq R_y^2,\quad i\in\mathcal I_T,\\
x^i=x^0-\eta\sum_{j=1}^{i}f^j-\eta\sum_{j=0}^{i-1}q^j,\quad
i\in\{1,\dots,T\},\\
y^i=y^0-\eta\sum_{j=1}^{i}h^j+\eta\sum_{j=1}^{i}p^j,\quad
i\in\{1,\dots,T\},\\
x^{-3}\ \text{and}\ y^{-3}\ \text{satisfy \eqref{eq:pep-wavg}},\\
(x^i)^\top q^j=(p^i)^\top y^j,\quad i,j\in\mathcal J_T,\\
L^2X^\top X-P^\top P\succeq0,\qquad
L^2Y^\top Y-Q^\top Q\succeq0.
\end{array}\right)
\label{eq:pep-nl}
\end{equation}
Every feasible point of \eqref{eq:pep-inner} corresponds to a feasible
point of \eqref{eq:pep-nl} with the same objective value.  The
constraints associated with AltGDA updates hold with the normal-cone elements supplied by
\cref{lem:pep-span}, the constraints modeling convex compact sets hold by
\cref{lem:pep-set-interp}, and the  operator constraint associated with modeling the matrix $A$ with $\|A\|_2 \leq L$ holds by
\cref{lem:matrix-interpolation}.  The samples at the reserved index
\(-1\) may be set to zero without violating any displayed constraint.

Conversely, consider any feasible point of \eqref{eq:pep-nl}.
By \cref{lem:pep-set-interp}, there are compact convex sets
\(\mathcal X\) and \(\mathcal Y\) of radii \(R_x\) and \(R_y\) that
contain the sampled points with the sampled normal-cone elements.
By \cref{lem:matrix-interpolation}, there is a matrix \(A\) with
\(\lVert A\rVert_2\leq L\) that matches every image pair, including
the samples at indices \(-1\) and \(-3\).
Finally, \cref{lem:pep-span} identifies the recursions with the AltGDA
projections while preserving the objective value.

Applying \(A\) and \(A^\top\) to the affine
combinations \eqref{eq:pep-wavg} shows that the realized data
satisfy \(p^{-3}=\sum_{j=1}^{T}j\,p^j/\sum_{j=1}^{T}j\) and
\(q^{-3}=\sum_{j=1}^{T}j\,q^j/\sum_{j=1}^{T}j\), so the averaged
sample is the weighted average of the corresponding trajectory.
Thus, the two optimal values are the same, and the equality in
\eqref{eq:pep-nl} is exact.

\subsection{Gram lift and the solved semidefinite program}
\label{subsec:pep-gram}

Stack the variables of \eqref{eq:pep-nl} into one data matrix per
player side,
\begin{equation}
\begin{aligned}
F^x&=\left[x^{-2}\mid x^{-1}\mid x^{0}\mid
f^{-3}\mid\cdots\mid f^{T}\mid
q^{-3}\mid\cdots\mid q^{T}\right]\in\R^{n\times(2T+11)},\\
F^y&=\left[y^{-2}\mid y^{-1}\mid y^{0}\mid
h^{-3}\mid\cdots\mid h^{T}\mid
p^{-3}\mid\cdots\mid p^{T}\right]\in\R^{m\times(2T+11)},
\end{aligned}
\label{eq:pep-data}
\end{equation}
in which the normal-cone block and the image block each run over
\(\mathcal J_T\) in the order \(-3,-2,-1,0,1,\dots,T\): three point
columns, \(T+4\) normal-cone columns, and \(T+4\) image columns per
side.  The Gram blocks are \(G^x=(F^x)^\top F^x\) and
\(G^y=(F^y)^\top F^y\), both in \(\mathbb S^{2T+11}_+\).  The Gram
matrices \(G^x\) and \(G^y\) arising this way have ranks at most \(n\) and
\(m\), respectively.
However, because in our worst-case optimization problem, we are also optimizing over $m,n$, we can drop the rank constraint without changing the optimal objective value.  Hence, the Gram description and optimizing over $m,n$ allows us to remove the
ambient dimension $m,n$ without changing the optimal value.

By the hatted symbols we denote coordinate selectors, where a selector
\(\widehat z\in\R^{2T+11}\) represents \(F^x\widehat z\) on the
\(x\)-side and \(F^y\widehat z\) on the \(y\)-side.  The columns of
\eqref{eq:pep-data} use basis selectors, while the remaining points use
affine selectors that encode the recursion and averaging:
\begin{equation}
\begin{aligned}
&\widehat x^{-2}=e_1,\quad
\widehat x^{-1}=e_2,\quad
\widehat x^{0}=e_3,\quad
\widehat f^{i}=e_{7+i},\quad
\widehat q^{i}=e_{T+11+i},\quad i\in\mathcal J_T,\\
&\widehat x^{i}=\widehat x^{0}
-\eta\sum_{j=1}^{i}\widehat f^{j}
-\eta\sum_{j=0}^{i-1}\widehat q^{j},
\quad i\in\{1,\dots,T\},
\qquad
\widehat x^{-3}
=\frac{\sum_{j=1}^{T}j\,\widehat x^{j}}{\sum_{j=1}^{T}j},
\end{aligned}
\label{eq:pep-selectors-x}
\end{equation}
\begin{equation}
\begin{aligned}
&\widehat y^{-2}=e_1,\quad
\widehat y^{-1}=e_2,\quad
\widehat y^{0}=e_3,\quad
\widehat h^{i}=e_{7+i},\quad
\widehat p^{i}=e_{T+11+i},\quad i\in\mathcal J_T,\\
&\widehat y^{i}=\widehat y^{0}
-\eta\sum_{j=1}^{i}\widehat h^{j}
+\eta\sum_{j=1}^{i}\widehat p^{j},
\quad i\in\{1,\dots,T\},
\qquad
\widehat y^{-3}
=\frac{\sum_{j=1}^{T}j\,\widehat y^{j}}{\sum_{j=1}^{T}j}.
\end{aligned}
\label{eq:pep-selectors-y}
\end{equation}
When the recursion and averaging constraints of \eqref{eq:pep-nl}
hold, \(x^i=F^x\widehat x^{i}\), \(f^i=F^x\widehat f^{i}\), and
\(q^i=F^x\widehat q^{i}\) for every \(i\in\mathcal J_T\), and
likewise \(y^i=F^y\widehat y^{i}\), \(h^i=F^y\widehat h^{i}\), and
\(p^i=F^y\widehat p^{i}\); every constraint and objective of
\eqref{eq:pep-nl} then becomes a trace or semidefinite condition on
\((G^x,G^y)\).  Define the selector matrices with columns ordered
\(-3,-2,-1,0,1,\dots,T\):
\[
\widehat X=[\widehat x^{i}]_{i\in\mathcal J_T},\quad
\widehat P=[\widehat p^{i}]_{i\in\mathcal J_T},\quad
\widehat Y=[\widehat y^{i}]_{i\in\mathcal J_T},\quad
\widehat Q=[\widehat q^{i}]_{i\in\mathcal J_T},
\]
all in \(\R^{(2T+11)\times(T+4)}\), so that \(X=F^x\widehat X\),
\(Q=F^x\widehat Q\), \(Y=F^y\widehat Y\), and \(P=F^y\widehat P\).
The resulting semidefinite program is the one solved in the study,
with the objective \(\mathcal M_G^{\mathrm{lst}}\) written out in
\cref{subsec:pep-objectives}:
\begin{equation}
\mathcal P_{T}(\eta)=\left(\begin{array}{l}
\underset{G^x,\ G^y\in\mathbb S^{2T+11}}{\mbox{maximize}}\;
\mathcal M_G^{\mathrm{lst}}\\[0.3em]
\textup{subject to}\\
\mathbf{tr}\!\left(G^x\boxdot(\widehat f^{j},
\widehat x^{i}-\widehat x^{j})\right)\leq0,\quad
i,j\in\mathcal I_T,\ i\neq j,\\
\mathbf{tr}\!\left(G^y\boxdot(\widehat h^{j},
\widehat y^{i}-\widehat y^{j})\right)\leq0,\quad
i,j\in\mathcal I_T,\ i\neq j,\\
\mathbf{tr}\!\left(G^x\boxdot(\widehat x^{i},\widehat q^{j})\right)
=\mathbf{tr}\!\left(G^y\boxdot(\widehat p^{i},\widehat y^{j})\right),
\quad i,j\in\mathcal J_T,\\
L^2\widehat X^\top G^x\widehat X
-\widehat P^\top G^y\widehat P\succeq0,\\
L^2\widehat Y^\top G^y\widehat Y
-\widehat Q^\top G^x\widehat Q\succeq0,\\
\mathbf{tr}\!\left(G^x\boxdot(\widehat x^{i},\widehat x^{i})\right)
\leq R_x^2,\quad i\in\mathcal I_T,\\
\mathbf{tr}\!\left(G^y\boxdot(\widehat y^{i},\widehat y^{i})\right)
\leq R_y^2,\quad i\in\mathcal I_T,\\
G^x\succeq0,\qquad G^y\succeq0
\end{array}\right).
\label{eq:pep-sdp}
\end{equation}

\paragraph{Equivalence between \eqref{eq:pep-inner}  and \eqref{eq:pep-sdp}}

Now we argue why \eqref{eq:pep-inner} and \eqref{eq:pep-sdp} are equivalent to each other with same optimal objective value. First note that every feasible point of \eqref{eq:pep-nl} maps to the pair
\(G^x=(F^x)^\top F^x\), \(G^y=(F^y)^\top F^y\), which is feasible for
\eqref{eq:pep-sdp} with the same objective value; the two objectives
correspond through the operator equality at \((i,j)=(T,-2)\).

On the other hand, factorizing a feasible pair of \eqref{eq:pep-sdp}
produces interpolation data that satisfy every constraint of \eqref{eq:pep-nl}. To see that, first note that
the algorithm update equations and averaging constraints hold because the point selectors are the stated affine combinations, and the remaining
families are the displayed trace and semidefinite conditions. As a result,  the
optimal value of \eqref{eq:pep-sdp} is equal to the modeled
worst case at every horizon. Hence,  \eqref{eq:pep-sdp} is equivalent to \eqref{eq:pep-inner} and our reformulation is tight.

Finally, we make some comments regarding the structure and patterns of  \eqref{eq:pep-sdp}. Note that neither \(m\) nor \(n\) appears in \eqref{eq:pep-sdp}. Also, the
normal-cone and radius families run over \(\mathcal I_T\), and the
three operator families run over all of \(\mathcal J_T\), including
the reserved index.  On each player side, the optimization problem
has \((T+3)(T+2)\) normal-cone inequalities and \(T+3\) radius
bounds; the sides share \((T+4)^2\) equalities, two semidefinite
constraints of order \(T+4\), and the two Gram blocks of order
\(2T+11\), which is \(71\) at the largest horizon \(T=30\). The averaging of the iterates is enforced directly through the selectors:
\(\widehat x^{-3}\) and \(\widehat y^{-3}\) are affine combinations of iterate
selectors, while \(\widehat p^{-3}\) and \(\widehat q^{-3}\) are
basis columns.  Finally, we note that
\(p^{-3}=\sum_{j=1}^{T}j\,p^j/\sum_{j=1}^{T}j\) and
\(q^{-3}=\sum_{j=1}^{T}j\,q^j/\sum_{j=1}^{T}j\) need no constraint
of their own as the matrix supplied by \cref{lem:matrix-interpolation}
enforces those conditions directly.

\subsection{The performance measure}
\label{subsec:pep-objectives}

The measure enters \eqref{eq:pep-sdp} through \(G^y\):
\begin{equation}
\mathcal M_G^{\mathrm{lst}}
=\mathbf{tr}\bigl(G^y\boxdot(\widehat p^{T},
\widehat y^{-2})\bigr)
-\mathbf{tr}\bigl(G^y\boxdot(\widehat p^{-2},
\widehat y^{T})\bigr).
\label{eq:pep-objective}
\end{equation}
At Gram pairs arising from data as in \eqref{eq:pep-data}, this
evaluates to \((y^\star)^\top Ax^T-(y^T)^\top Ax^\star\), the
objective of \eqref{eq:pep-inner}.  The exact optimal value under this
objective is \(\mathcal P_T(\eta)\), which equals \(2\) when
\(L=R_x=R_y=1\) by \cref{prop:pep-exact-plateau}.
The numerical values in \cref{fig:pep-lastiterate} confirm this value
under the solver tolerances in \cref{subsec:pep-protocol}.

\subsection{Numerical PEP Setup}
\label{subsec:pep-protocol}

With \(L=R_x=R_y=1\), the PEP search evaluates each horizon
\(T=5,\ldots,30\) at \(25\) logarithmically spaced 
stepsizes \(\eta_c=1/(\eta L)\) in \([0.5,64]\), corresponding to
\(\eta\in[1/64,2]\). 

For every numerical experiment for this section, we have used MOSEK \(11.0.29\) through
JuMP \citet{Lubin2023} with
\texttt{MSK\_DPAR\_INTPNT\_CO\_TOL\_PFEAS}
\(=\)~\texttt{MSK\_DPAR\_INTPNT\_CO\_TOL\_DFEAS} \(=10^{-4}\) and
automatic thread selection.  A solve is accepted only with
termination status \texttt{OPTIMAL} and primal and dual solution statuses
\texttt{FEASIBLE\_POINT}; no solve in our numerical study was rejected.

\subsection{Worst-case examples for last-iterate duality gap}
\label{subsec:pep-exact-witnesses}

\begin{proof}[Proof of \cref{prop:pep-exact-plateau}]
For any feasible instance of \eqref{eq:pep-inner}, any comparator point
\((x^\star,y^\star)\), and any \(T\geq0\), the norm and radius bounds give us:
\[
(y^\star)^\top Ax^T-(y^T)^\top Ax^\star
\leq \lVert y^\star\rVert\lVert A\rVert_2\lVert x^T\rVert
   +\lVert y^T\rVert\lVert A\rVert_2\lVert x^\star\rVert
\leq 2.
\]
Taking the maximum over comparators and instances proves the upper
bound.  To attain this upper bound, let \(A=I_2\) and let both feasible sets be the
closed Euclidean unit ball in \(\R^2\).  Then their duality gap satisfies
\[
\Gap(x,y)=\max_{\lVert v\rVert\leq1}v^\top x
          -\min_{\lVert u\rVert\leq1}y^\top u
        =\lVert x\rVert+\lVert y\rVert.
\]
Now consider  \(0<\eta\leq2\), and initialize
\[
y^0=(1,0), \text{ and }
x^0=\left(\frac\eta2,\sqrt{1-\frac{\eta^2}{4}}\right).
\]
These points satisfy \(\lVert x^0\rVert=\lVert y^0\rVert=1\) and
\(\langle x^0,y^0\rangle=\eta/2\).  Suppose now the same identities hold
at an iterate \((x,y)\), and form the unprojected candidates
\(x'=x-\eta y\) and \(y'=y+\eta x'\).  Direct calculation gives us:
\[
\begin{aligned}
\lVert x'\rVert^2
 &=1-2\eta(\eta/2)+\eta^2=1,
 &\langle x',y\rangle&=-\eta/2,\\
\lVert y'\rVert^2
 &=1+2\eta(-\eta/2)+\eta^2=1,
 &\langle x',y'\rangle&=-\eta/2+\eta=\eta/2.
\end{aligned}
\]
Both candidates lie in their unit balls, so the projections leave them
unchanged.  Using induction we can easily prove the identities and hence gap \(2\) at
every iterate, including the endpoint \(\eta=2\).

Now consider the stepsize range \(\eta>2\), and use the same matrix and sets with
\(x^0=y^0=(1,0)\).  If \(x^t=y^t=e\), where
\(e\in\{(1,0),(-1,0)\}\), the first unprojected candidate is
\((1-\eta)e\), whose projection is \(-e\).  The second candidate is
then also \((1-\eta)e\), with projection \(-e\).  Thus both players are
essentially alternating between \((1,0)\) and \((-1,0)\), retaining gap \(2\).
These attain the upper bound at initialization and every
subsequent horizon.  The first initialization depends on \(\eta\),
as allowed in \eqref{eq:pep-inner}, while both constructions are using fixed
two-dimensional sets and a fixed matrix.  Therefore
\(\mathcal P_T(\eta)=2\) for every \(\eta>0\) and \(T\geq1\), and
\(\inf_{\eta>0}\mathcal P_T(\eta)=2\).
\end{proof}

\begin{proof}[Proof of \cref{ex:simplex-last-iterate}]
Write the simplex iterates as
\[
x^t=(1/2+a_t,1/2-a_t),\qquad
y^t=(1/2+b_t,1/2-b_t),
\]
where \(a_t,b_t\in[-1/2,1/2]\).  The prescribed initialization gives
\(a_0=1/4\), \(b_0=0\) for every stepsize.  Since
\(A^\top y^t=(2b_t,-2b_t)\) and
\(Ax^{t+1}=(2a_{t+1},-2a_{t+1})\), the updates reduce to
\[
a_{t+1}=\operatorname{clip}(a_t-2\eta b_t),\qquad
b_{t+1}=\operatorname{clip}(b_t+2\eta a_{t+1}),
\]
where \(\operatorname{clip}(u)=\min\{1/2,\max\{-1/2,u\}\}\).
Note that each unprojected vector has coordinate sum one. As a result the simplex projection clips the corresponding centered coordinate to
\([-1/2,1/2]\).  Also, maximizing and minimizing the payoff coordinates gives
\(\Gap(x^t,y^t)=2(|a_t|+|b_t|)\).

First, let us consider \(0<\eta<1\), and define
\(E_t=a_t^2+b_t^2-2\eta a_tb_t\).  If neither update is clipped,
substitution in the unprojected recurrence gives \(E_{t+1}=E_t\).
If at least one update is clipped, the new point has
\(|a_{t+1}|=1/2\) or \(|b_{t+1}|=1/2\).  Completing the square gives
\[
E_{t+1}=(b_{t+1}-\eta a_{t+1})^2
       +(1-\eta^2)a_{t+1}^2
\geq\frac{1-\eta^2}{4}
\]
when \(|a_{t+1}|=1/2\); interchanging the two coordinates gives the
same bound when \(|b_{t+1}|=1/2\).  Starting from \(E_0=1/16\),
induction therefore yields
\[
E_t\geq\min\left\{\frac1{16},\frac{1-\eta^2}{4}\right\}
\qquad(t\geq0).
\]
Since \(\eta<1\), we also have
\(E_t\leq(|a_t|+|b_t|)^2\).  Thus
\[
\Gap(x^t,y^t)\geq2\sqrt{E_t}
\geq\min\left\{\frac12,\sqrt{1-\eta^2}\right\}>0
\qquad(t\geq0).
\]

For \(\eta\geq1\), the same initialization produces
\[
(a_0,b_0)=(1/4,0),\qquad
(a_1,b_1)=(1/4,1/2),\qquad
(a_2,b_2)=(-1/2,-1/2).
\]
The clipped recurrence then alternates between
\((-1/2,-1/2)\) and \((1/2,1/2)\).  The gaps are consequently
\(1/2\) at \(t=0\), \(3/2\) at \(t=1\), and \(2\) for every
\(t\geq2\).  These identities include the endpoint \(\eta=1\).
\end{proof}

\begin{remark}[Averaged convergence on the same trajectory]
\label{rem:simplex-uniform-smallsteps}
Note that for every fixed \(0<\eta\leq1/2\), the original initialization
\(x^0=(3/4,1/4)\), \(y^0=(1/2,1/2)\) yields last-iterate gaps at
least \(1/2\) implying non-convergence, while its uniform averages converge.
To see the latter directly, we first define the unprojected recurrence
\(a_{t+1}=a_t-2\eta b_t\),
\(b_{t+1}=b_t+2\eta a_{t+1}\) from \(a_0=1/4\), \(b_0=0\).
It preserves \(E_t=1/16\), and
\(E_t\geq(1-\eta)(a_t^2+b_t^2)\) gives
\[
|a_t|,|b_t|\leq\frac{1}{4\sqrt{1-\eta}}
\leq\frac{1}{2\sqrt2}<\frac12.
\]
Using induction we can thus show that every projection leaves its candidate
unchanged thus corresponding to an admissible AltGDA trajectory.  Telescoping the
two scalar updates over the averaging indices gives us:
\[
\frac1T\sum_{t=1}^T a_t=\frac{b_T-b_0}{2\eta T},\qquad
\frac1T\sum_{t=1}^T b_t=\frac{a_1-a_{T+1}}{2\eta T}.
\]
Using \(b_0=0\), \(a_1=1/4\), and the coordinate bounds gives us:
\[
\Gap(\overline x^T,\overline y^T)
=\frac{|b_T-b_0|+|a_1-a_{T+1}|}{\eta T}
\leq\frac{5}{4\eta T}.
\]
The last-iterate lower bound \(1/2\) is uniform in both the horizon
and \(\eta\in(0,1/2]\), with the matrix and initialization fixed.
It therefore also applies when a different constant stepsize in this
range is selected for each horizon.
The matrix has eigenvalues \(0\) and \(2\), so \(L=2\), and its
unique saddle point \(x^\star=y^\star=(1/2,1/2)\) has full supports,
giving \(\delta=1/2\) in \eqref{eq:delta-main}.
The general ergodic theorem, \cref{thm:main-text}, also applies for
\(0<\eta\leq1/(8\sqrt2)\) and the larger range \((0,1/2]\) proved
here is specific to this example.
\end{remark}

\end{document}